\documentclass[reqno]{amsart}
\usepackage[noadjust]{cite} 
\usepackage{xcolor}
\usepackage{caption}
\usepackage[figuresright,counterclockwise]{rotating} 
\usepackage{graphicx}
\usepackage{subfigure}
\usepackage{multirow}
\usepackage{booktabs}
\usepackage{float}
\usepackage[colorlinks,citecolor=red,linkcolor=blue,hyperindex,CJKbookmarks]{hyperref}
\usepackage[shortlabels]{enumitem}
\setenumerate{itemsep=0pt,partopsep=0pt,parsep=\parskip,topsep=5pt}
\setitemize{itemsep=0pt,partopsep=0pt,parsep=\parskip,topsep=5pt}
\setdescription{itemsep=0pt,partopsep=0pt,parsep=\parskip,topsep=5pt}
\usepackage{algorithm,algorithmic}
\usepackage{amsmath,amsthm,amssymb}
\usepackage{mathtools,mathrsfs}
\usepackage{mleftright} 

\renewcommand{\left}{\mleft} 
\renewcommand{\right}{\mright}
\usepackage{cases}
\usepackage{dsfont}
\theoremstyle{plain}
\newtheorem{theorem}{Theorem}[section] %
\newtheorem{lemma}[theorem]{Lemma} %
\theoremstyle{definition}
\newtheorem{definition}[theorem]{Definition} %
\newtheorem{assumption}[theorem]{Assumption} %
\newtheorem{example}{Example} %
\theoremstyle{remark} %
\numberwithin{equation}{section}
\newcommand{\norm}[1]{\left\Vert#1\right\Vert} 
\newcommand{\normF}[1]{\left\Vert#1\right\Vert_F} 
\newcommand{\abs}[1]{{\left |#1\right |}}

\newcommand{\card}{\mathrm{card}} 

\newcommand{\R}{\mathbb{R}} 
\renewcommand{\P}{\mathbb{P}} 
\newcommand{\E}{\mathbb{E}} 
\newcommand{\kbf}{\kappa_{\mathrm{bf}}}

\newcommand{\kbv}{\kappa_{\mathrm{bv}}}
\newcommand{\kbmj}{\kappa_{\mathrm{bmj}}}
\newcommand{\kej}{\kappa_{\mathrm{ej}}}
\newcommand{\klf}{\kappa_{\mathrm{lf}}}
\newcommand{\klj}{\kappa_{\mathrm{lj}}}
\newcommand{\klg}{\kappa_{\mathrm{lg}}}
\newcommand{\Jmk}{J_{m_k}}
\newcommand{\ITk}{\mathbf{1}_{T_k}}

\begin{document}
\title[Derivative-free LM algorithm via $ \ell_1 $ minimization]{A derivative-free Levenberg-Marquardt method for sparse nonlinear least squares problems}

\date{}

\author{Yuchen Feng}
\address{School of Mathematical Sciences, Shanghai Jiao Tong University, Shanghai 200240, China}
\email{fengyuchen@sjtu.edu.cn}

\author{Chuanlong Wang}
\address{Shanxi Key Laboratory for Intelligent Optimization Computing and Block-chain Technology, Department of Mathematics, Taiyuan Normal University, Jinzhong 030619, Shanxi Province, China}
\email{clwang1964@163.com}

\author{Jinyan Fan*}
\address{School of Mathematical Sciences, and MOE-LSC, Shanghai Jiao Tong University, Shanghai 200240, China}
\email{jyfan@sjtu.edu.cn}

\thanks{* Corresponding author}

\thanks{The authors are supported by the Science and Technology Commission of Shanghai Municipality
	grant 22JC1401500, the National Natural Science Foundation of China grants 12371307 and 12371381, and
	the Fundamental Research Funds for the Central Universities.}


\begin{abstract}
	This paper studies sparse nonlinear least squares problems, where the Jacobian matrices are unavailable or expensive to compute, yet have some underlying sparse structures.
	We construct the Jacobian models by the $ \ell_1 $ minimization subject to a small number of interpolation constraints with interpolation points generated from some certain distributions,
	and propose a derivative-free Levenberg-Marquardt algorithm based on such Jacobian models.
	It is proved that the Jacobian models are probabilistically first-order accurate
	and the algorithm converges globally almost surely.
	Numerical experiments are presented to show the efficiency of the proposed algorithm for sparse nonlinear least squares problems.
\end{abstract}

\maketitle

\section{Introduction}
Nonlinear least squares problems arise frequently in areas of engineering, economy, machine learning and so on.
	In this paper, we consider the nonlinear least squares problems
	\begin{equation}\label{eq: nonlinear least squares problems}
		\min_{x\in\R^n} \ f(x)=\frac{1}{2}\norm{F(x)}^2 = \frac{1}{2}\sum_{i=1}^{m}\big(F^{(i)}(x)\big)^2,
	\end{equation}
	where $ F(x): \R^n \rightarrow  \R^m $ is continuously differentiable, yet its Jacobian $ J(x): \R^n\rightarrow \R^{m\times n} $ is either unavailable or expensive to compute but has some underlying sparse structure. Thus, the nonlinear least squares methods that rely on exact Jacobian matrices are not applicable anymore.

For derivative-free nonlinear least squares problems,
Brown et al. \cite{brown1971DerivativeFreeAnalogues} gave finite difference analogues of the Levenberg-Marquardt (LM) method and Gauss-Newton method.
Some model-based methods were also developed.
For example, Zhang et al. \cite{zhang2012LocalConvergenceDerivativefree} built quadratic interpolation models for each residual function $F^{(i)}$ and presented a framework of derivative-free algorithms within a trust region framework, while Cartis et al. \cite{cartis2019DerivativefreeGaussNewton,cartis2019ImprovingFlexibilityRobustness} constructed linear interpolation/regression models for each $F^{(i)}$ and presented a derivative-free version of the Gauss-Newton method.
Different from the above deterministic methods, probabilistic methods have attracted increasing attention in recent years.
Bergou et al. \cite{Bergou2016LM} presented a LM method based on probabilistic gradient models.
Zhao et al. \cite{zhao2022LevenbergMarquardtMethod} established linear interpolation approximation for each $F^{(i)}$ with the interpolation points generated by the standard Gaussian distribution, and showed that the Jacobian models obtained are probabilistically first-order accurate and the LM method based on such Jacobian models converges to a first-order stationary point of the objective function almost surely.
Chen et al. \cite{chen2024GlobalComplexityDerivativeFree} built the probabilistically first-order accurate Jacobian models via orthogonal spherical smoothing and proved that the LM method based on such Jacobian models also converges globally almost surely.
In \cite{cartis2023ScalableSubspaceMethods}, Cartis et al. introduced a general framework for large-scale model-based derivative-free optimization based on iterative minimization within random subspaces.
The framework is then specialized to nonlinear least-squares problems, which is based on the Gauss-Newton method and constructs local linear interpolation models to approximate the Jacobian.

In some practical applications, each residual function $F^{(i)}$ involves only a small fraction of variables, implying a sparse structure in the Jacobian matrix. Typical problems of this kind are large scale geodesic adjustment \cite{golub1980LargescaleGeodeticLeastsquares}, least squares network adjustment \cite{teunissen2000AdjustmentTheoryIntroduction}, wireless sensor network localization \cite{Mao2007WirelessSN}, etc.
Curtis et al. \cite{curtis1974EstimationSparseJacobian} showed how to reduce the work required to estimate the Jacobian matrix by finite differences when it is sparse with a known sparsity pattern.
An et al. \cite{an2007GloballyConvergentNewtonGMRES} presented an inexact Newton methods with quasi-conjugate-gradient backtracking for large sparse systems of nonlinear equations.
Kreji{\'c} et al. \cite{krejic2023SplitLevenbergMarquardtMethod} proposed a decoupling procedure which splits the original problems into a sequence of independent problems of smaller sizes and developed a modification of the LM method with smaller computational costs. However, all these methods require either the sparsity pattern of the Jacobian matrices or the exact derivatives.

In this paper, we focus on nonlinear least squares problems, where the Jacobian matrices are unavailable and possess some underlying sparse structures but the sparsity patterns are not known.
We construct the Jacobian models based on compressed sensing and propose a derivative-free LM algorithm for solving it.

First, we show how to construct the sparse Jacobian models by linear interpolation and
$\ell_1 $ minimization, then propose a derivative-free LM algorithm based on such Jacobian models in Section \ref{sec: algorithm}.
Second, we introduce the probabilistically bounded RIP matrices (i.e., satisfying the restricted isometry property), and prove that the sparse Jacobian models constructed are first-order accurate if the interpolation points are generated appropriately, such as from Gaussian distribution, Bernoulli distribution and Bernoulli-like distribution.
These results are given in Section \ref{sec: Jacobian models}.
Third, we study convergence properties of the algorithm and
show each accumulation point of the sequence generated by the algorithm is a stationary point of the objective function with probability one in Section \ref{sec: convergence}.
In Section \ref{sec: numerical results}, we present numerical experiments for solving sparse nonlinear least squares problems. Numerical results indicate that the number of interpolation points can be much less than the problem dimension if the Jacobian sparsity level is high.
Some preliminaries in compressed sensing are given in Section \ref{sec: Jacobian models}.

Notations: For a vector $ x\in\R^n $, $ \|x\|_0 $ is the number of nonzero entries, $ \|x\|_1 =\sum_{i=1}^{n}|x^{(i)}| $ is the $ \ell_1 $ norm, and $ \|x\| $ is the standard Euclidean norm. For a matrix $ X\in\R^{m\times n} $, $ \|X\| $ is the spectral norm, and $ \|X\|_F $ is the Frobenius norm. $ B(x,r) $ denotes the closed ball $ B(x,r):=\{y\in\R^n \mid \norm{y-x}\leq r\} $.

\section{Derivative-free LM algorithm for sparse nonlinear least squares problems}\label{sec: algorithm}

In this section, we show how to construct the sparse Jacobian models by linear interpolation and $\ell_1 $ minimization, and then propose a LM algorithm for sparse nonlinear least squares problems based on such Jacobian models.


At the $k$-th iteration, we generate $p$ random vectors $v^1_k,\ldots,v^p_k\in \R^n$ with $p<n$ and take the set of interpolation points
\begin{equation}
	\{x_k+\sigma_k v^1_k,\ldots,x_k+\sigma_k v^p_k\},
\end{equation}
where $\sigma_k>0$ is updated at each iteration and controls the distance from the interpolation points to $x_k$. Throughout the paper, we denote $ v_k = \{v^1_k,\ldots,v^p_k\}$.

	For each residual function $F^{(i)}(x), i = 1, \ldots, m$, construct its linear interpolation model $ m_k^{(i)}(x)=c_k^{(i)}+(g_k^{(i)})^T(x-x_k) $ that satisfies
	\begin{subequations}\label{eq: interp condition}
		\begin{align}
			m_k^{(i)}(x_k) & =F^{(i)}(x_k), \label{eq: interp condition a} \\
			m_k^{(i)}(x_k+\sigma_k v_k^j) & =F^{(i)}(x_k+\sigma_k v_k^j),\quad j = 1,\ldots,p. \label{eq: interp condition b}
		\end{align}
	\end{subequations}
	Substituting \eqref{eq: interp condition a} into \eqref{eq: interp condition b} gives
	\begin{equation}\label{eq: interp condition shifted}
		(g_k^{(i)})^T(\sigma_k v_k^j) = F^{(i)}(x_k+\sigma_k v_k^j)-F^{(i)}(x_k),\quad j =1,\ldots,p.
	\end{equation}
	Denote
	\begin{equation}\label{eq: interp condition matrix form}
		A(v_k)=\begin{bmatrix}
			(v^1_k)^T \\ \vdots \\ (v^p_k)^T
		\end{bmatrix}\in \R^{p\times n}, \quad y_k^{(i)}=\begin{bmatrix}
			F^{(i)}(x_k+\sigma_k v_k^1) -F^{(i)}(x_k) \\
			\vdots \\
			F^{(i)}(x_k+\sigma_k v_k^p) -F^{(i)}(x_k)
		\end{bmatrix}\in \R^{p}.
	\end{equation}
	\eqref{eq: interp condition shifted} can be written as
	\begin{equation}
		A(v_k) g_k^{(i)} = \sigma_k^{-1}y_k^{(i)}, \quad i=1,\ldots,m.
	\end{equation}

Note that the Jacobians of the problem \eqref{eq: nonlinear least squares problems} are sparse.
One can solve the $ \ell_1 $ minimization problem
\begin{equation}\label{prob: l1-Jacobian-noisy}
	\begin{aligned}
		\min_{g\in\R^n}\ & \norm{g}_1 \\
		\text{s.t.}\ & \norm{A(v_k)g - \sigma_k^{-1}y_k^{(i)} }\leq\xi_k,
	\end{aligned}
\end{equation}
to obtain $ g_k^{(i)}$, where $\xi_k$ is a positive number, and take the Jacobian model $\Jmk$ as
\begin{equation}\label{jacobianmodel}
	\Jmk=\begin{bmatrix}
		\big(g^{(1)}_k\big)^T \\ \vdots \\ \big(g^{(m)}_k\big)^T
	\end{bmatrix}\in \R^{m\times n}.
\end{equation}
Then $ \Jmk $ can serve as an approximation of $ J(x_k) $. The choices of $p,~ v_k,~\sigma_k$ and $\xi_k$ will be discussed in Section \ref{sec: Jacobian models}.

Denote $ F_k = F(x_k)$. We solve the system of linear equations
\begin{equation}\label{eq: lm equation}
	(\Jmk^T\Jmk+\lambda_kI)d=-\Jmk^TF_k \text{ with } \lambda_k=\theta_k\norm{\Jmk^TF_k}
\end{equation}
to obtain the trial step $ d_k$, where $\lambda_k$ is the LM parameter with $\theta_k$ being updated from iteration to iteration.
Define the ratio of the actual reduction to the predicted reduction of the objective function as
\begin{equation}\label{eq: ratio}
	\rho_k=\frac{\mathrm{Ared}_k}{\mathrm{Pred}_k}=\frac{\norm{F_k}^2-\norm{F(x_k+d_k)}^2}{\norm{F_k}^2-\norm{F_k+\Jmk d_k}^2}.
\end{equation}
It plays an important role in deciding whether $ d_k $ is acceptable and how to update $ \theta_k $.
If $\rho_k$ is smaller than a small positive constant, we reject $d_k$ and increase $\theta_k$. Otherwise we accept $d_k$ and update $\theta_k$ as in \cite{zhao2018LevenbergMarquardtMethod}: we decrease $\theta_k$ if the product of $\norm{\Jmk^TF_k}$ and $\theta_k$ is larger than a positive constant, otherwise we keep it unchanged or increase it.

The derivative-free LM algorithm for sparse nonlinear least squares problems is presented as follows.

\begin{algorithm}[H]
	\renewcommand{\algorithmicrequire}{\textbf{Input:}}
	\caption{Derivative-free LM algorithm for sparse nonlinear least squares problems.}
	\label{algo: sparseDFOLM}
	\begin{algorithmic}[1]
		\REQUIRE $0 <\eta_0<1, 0<\eta_1<\eta_2, 0<\gamma_1<1<\gamma_2 $,
		$ x_0\in\R^n, \theta_0 \geq \theta_{\min} > 0, \varepsilon_0>0$.
		\FOR{$ k=0,1,\ldots $}
		\STATE Construct the sparse Jacobian model $ \Jmk $ by \eqref{prob: l1-Jacobian-noisy}--\eqref{jacobianmodel}.
		\STATE If {$\norm{\Jmk^TF_k}\leq\varepsilon_0$}, stop. Otherwise solve \eqref{eq: lm equation} to obtain $ d_k $.
		\STATE Compute $ \rho_k $ by \eqref{eq: ratio}.
		\STATE	Set
		\begin{equation}\label{algo-eq: update xk}
			x_{k+1}=\left\{ \begin{array}{ll}
				x_k+d_k,\quad & \text{if}\; \rho_k>\eta_0, \\
				x_k, & \text{otherwise}.
			\end{array}\right.
		\end{equation}
		\STATE If $ \rho_k < \eta_0 $, set $ \theta_{k+1} = \gamma_2 \theta_{k}$. Otherwise compute
		\begin{equation}\label{algo-eq: update LM para}
			\theta_{k+1}=\left\{ \begin{array}{ll}
				\gamma_2\theta_k, & \text{if}\; \norm{\Jmk^TF_k}<\frac{\eta_1}{\theta_k}, \\[1mm]
				\theta_k, & \text{if}\; \norm{\Jmk^TF_k}\in[\frac{\eta_1}{\theta_k},\frac{\eta_2}{\theta_k}], \\[1mm]
				\max\{\gamma_1\theta_k, \theta_{\min}\},\quad & \text{if}\; \norm{\Jmk^TF_k}>\frac{\eta_2}{\theta_k}.
			\end{array}\right.
		\end{equation}
		\ENDFOR
	\end{algorithmic}
\end{algorithm}

The randomness of the interpolation points implies the randomness of the Jacobian models, the iteration points, the LM parameters and so on.
Consider a random model of the form $M_k$.
We use the notation $ \Jmk = J_{M_k}(V_k) $, $ x_k=X_k(V_k)$ and $ \theta_k= \Theta_k(V_k)$ for their realizations.

\section{First-order accurate Jacobian models}\label{sec: Jacobian models}
In this section, we first review some basic results in the field of compressed sensing, which are very helpful for the later analysis, then prove that the sparse Jacobian models constructed as described in Section \ref{sec: algorithm} are first-order accurate if $ v_k $, $ \sigma_k $ and $ \xi_k $ are taken appropriately.

The focus of noisy compressed sensing is to solve the underdetermined system
\begin{equation}
	y=A\hat{x}+z,
\end{equation}
where $ \hat{x}\in \R^n $ is the sparse signal to reconstruct, $ y \in \R^m $ is the available measurement with $ m\ll n $, $ A\in \R^{m\times n} $ is the known sensing matrix, and $ z $ is the unknown noise.
The recovery of sparse signals relies on the restricted isometry property (RIP) of the sensing matrix (cf. \cite{candes2005DecodingLinearProgramming}).

\begin{definition}\label{def: RIP}
	For each integer $s = 1, 2, \ldots$, define the $ s $-RIP constant $\delta_s$ of a matrix $A$ as the smallest number such that
	\begin{equation}\label{def-eq: RIP}
		(1-\delta_s)\norm{x}^2\leq \norm{Ax}^2 \leq (1+\delta_s)\norm{x}^2
	\end{equation}
	holds for all $s$-sparse vectors. A vector is said to be $s$-sparse if it has at most $s$ nonzero entries.
\end{definition}

The following theorem shows that the desired sparse vector $ \hat{x} $ can be recovered by solving an $ \ell_1 $ minimization problem, provided that the RIP constant of the sensing matrix satisfies some certain condition.

\begin{theorem}[{\cite[Theorem 4.6]{mo2011NewBoundsRestricted}}]
	\label{thm: noisy recovery mo2011}
	Let $ y=A\hat{x}+z $, where $ \hat{x} $ is $ s $-sparse and $\norm{z}\leq\xi$. If the $ 2s $-RIP constant of $ A $ satisfies $\delta_{2s}<0.4931$ , then the solution $ x^* $ to
	\begin{equation}\label{prob: l1-noisy Ax=b}
		\begin{aligned}
			\min_{x\in\R^n} \ & \norm{x}_1 \\
			\rm{s.t.}\ \ & \norm{Ax-y}\leq \xi
		\end{aligned}
	\end{equation}
	satisfies
	\begin{equation}
		\norm{x^*-\hat{x}}\leq c_1\xi,
	\end{equation}
	where $ c_1 $ is a constant only depends on $ \delta_{2s} $.
\end{theorem}

For simplicity, we use the RIP matrix as a shorthand for the matrix that satisfies the RIP of order $ 2s $ with constant $ \delta_{2s}<0.4931 $, where $ s $ is the sparsity level.

Define the event
\begin{equation}\label{Tk}
	T_k=\{ A(v_k) \in \R^{p\times n} \ \mbox{is an RIP matrix and}\ \|v_k^j\| \leq \kbv \ \mbox{for all} \ j=1,\ldots,p \,\},
\end{equation}
where $A(v_k)$ is defined by \eqref{eq: interp condition matrix form} and $\kbv$ is a positive constant.

\begin{definition}\label{def: RIP and bounded}
	Given $ \beta\in(0,1) $ and $ \kbv>0 $. $ \{A(V_k)\} $ is said a sequence of $ \beta $-probabilistically $ \kbv $-bounded RIP matrices if the events $ T_k $ satisfy the submartingale-like condition
	\begin{equation}\label{def-eq: Tk submartingale}
		\P\left(T_k \mid \mathcal{F}_{k-1}\right) \geq \beta,
	\end{equation}
	where $ \mathcal{F}_{k-1} $ is the $ \sigma $-algebra generated by $ V_0,\ldots,V_{k-1} $.
\end{definition}

To estimate the difference between the exact Jacobian matrix and its model, we introduce the definition of the first-order accurate Jacobian model as in \cite{zhao2022LevenbergMarquardtMethod}. Hereafter, $ J_k $ is used to denote $ J(x_k) $.

\begin{definition}\label{def: Jacobian accurate}
	Given $ \kej>0 $. 	A Jacobian model realization $ \Jmk $ is said to be $\kej $-first-order accurate if
	\begin{equation}\label{def-eq: Jacobian accurate}
		\norm{\Jmk-J_k}\leq\frac{\kej}{\theta_k}.
	\end{equation}
\end{definition}

Now we prove that $ \Jmk $ constructed by \eqref{prob: l1-Jacobian-noisy}--\eqref{jacobianmodel} is $\kej $-first-order accurate if $ v_k, \sigma_k $ and $ \xi_k $ are chosen appropriately.
We make some assumptions.

	\begin{assumption}\label{assp: J Lipschitz}
		(i) $ F(x): \R^n \rightarrow \R^m $ is continuously differentiable and the gradients $ \nabla F^{(1)}(x), \ldots, \nabla F^{(m)}(x)$ are Lipschitz continuous, i.e., there exists $ \klg >0$ such that for all $ i=1,\ldots,m, $
		\begin{equation}\label{assp-eq: gradient Lipschitz}
			\norm{\nabla F^{(i)}(x)-\nabla F^{(i)}(y)}\leq \klg\norm{x-y},\quad \forall x,y \in \R^n.
		\end{equation}

		(ii) $ \nabla F^{(1)}(x), \ldots, \nabla F^{(m)}(x)$ are $ s $-sparse for all $x\in \R^n $.
	\end{assumption}

	\eqref{assp-eq: gradient Lipschitz} implies that there exists $ \klj > 0 $ such that
	\begin{equation}\label{assp-eq: jacobian Lipschitz}
		\norm{J(x)-J(y)}\leq \klj\norm{x-y},\quad \forall x,y \in \R^n.
	\end{equation}

	\begin{theorem}\label{thm: Jacobian accurate}
		Suppose that Assumption \ref{assp: J Lipschitz} holds.
		Let $\sigma_k=\norm{d_{k-1}}$ and $\xi_k =\sqrt{p}\klg \kbv^2\sigma_k/2 $.
		If $ T_k $ happens, then $ \Jmk $ constructed by \eqref{prob: l1-Jacobian-noisy}--\eqref{jacobianmodel} is $ \kej $-first-order accurate,
		where $\kej = \sqrt{mp} c_1\gamma_2 \klg \kbv^2/2$
		with $c_1$ given in Theorem \ref{thm: noisy recovery mo2011}.
	\end{theorem}
	\begin{proof}
		Since $\nabla F^{(i)}(x)$ is $\klg$-Lipschitz continuous and $ T_k $ happens, for $j=1,\ldots,p,$ we have
		\begin{align}
			& \hspace{1.3em} \abs{F^{(i)}(x_k+\sigma_k v^j_k)-F^{(i)}(x_k)-\nabla F^{(i)}(x_k)^T(\sigma_k v^j_k)} \nonumber\\
			& = \abs{\int_{0}^{1}\big[\nabla F^{(i)}(x_k+t\sigma_k v^j_k)-\nabla F^{(i)}(x_k)\big]^T (\sigma_k v^j_k) dt} \nonumber \\
			& \leq \int_{0}^{1} \klg\big\|\sigma_k v^j_k\big\|^2\;t\;dt \nonumber\\
			& \leq \frac12\klg \kbv^2\sigma_k^2,
		\end{align}
		which yields
		\begin{equation}
			\norm{ A(v_k)\nabla F^{(i)}(x_k)-\sigma_k^{-1}y_k^{(i)}} \leq \frac{1}{2}\sqrt{p}\klg \kbv^2\sigma_k.
		\end{equation}
		Then, by Theorem \ref{thm: noisy recovery mo2011},
		\begin{equation}\label{eq: alpha-nabla}
			\norm{g_k^{(i)}-\nabla F^{(i)}(x_k)} \leq \frac{1}{2}\sqrt{p}c_1\klg \kbv^2\sigma_k.
		\end{equation}
		Hence
		\begin{equation}
			\normF{J_{m_k}-J_k} \leq \frac12 \sqrt{mp}c_1 \klg \kbv^2 \sigma_k.
		\end{equation}

		By the definition of $d_k$ and \eqref{algo-eq: update LM para},
		\begin{equation}
			\sigma_k = \norm{d_{k-1}} \leq \frac{\big\|J_{m_{k-1}}^TF_{k-1}\big\|}{\lambda_{k-1} } = \frac{1}{\theta_{k-1}} \leq \frac{\gamma_2}{\theta_k}.
		\end{equation}
		Thus
		\begin{equation}\label{eq: Jmk-Jk-accurate}
			\norm{J_{m_k}-J_k}\leq\normF{J_{m_k}-J_k}\leq \frac{\sqrt{mp}c_1\gamma_2 \klg \kbv^2}{2\theta_k}.
		\end{equation}
		The proof is completed.
	\end{proof}

%

There are many types of random matrices such that $T_k$ happens.
For example, suppose that $A(v_k)\in\R^{p\times n}$ is generated from one of the following distributions:
\begin{enumerate}[(a)]
	\item Gaussian distribution:
	      \begin{equation}\label{distribution: Gaussian}
		      a_{i,j}\sim \mathcal{N}\big(0,\frac{1}{p}\big);
	      \end{equation}
	\item Bernoulli distribution:
	      \begin{equation}\label{distribution: Bernoulli}
		      \P\Big(a_{i,j}=\frac{1}{\sqrt{p}}\Big) = \P\Big(a_{i,j}=-\frac{1}{\sqrt{p}}\Big) = \frac{1}{2};
	      \end{equation}
	\item A Bernoulli-like distribution:
	      \begin{equation}\label{distribution: Bernoulli-like}
		      \P\Big(a_{i,j}=\sqrt{\frac{3}{p}}\Big) = \P\Big(a_{i,j}=-\sqrt{\frac{3}{p}}\Big) = \frac{1}{6},\quad \P\big(a_{i,j}=0\big) = \frac{2}{3}.
	      \end{equation}
\end{enumerate}
Then by {\cite[Theorem 5.2]{baraniuk2008SimpleProofRestricted}}, given $ \hat{\delta}\in (0,1)$, there exist $ a_1, a_2 >0$ depending only on $ \hat{\delta}$ such that if
\begin{equation}\label{eq: number of interpolation points}
	p\geq \frac{ 2s\log(n/2s)}{a_1},
\end{equation}
then the probability that $ A(v_k) $ satisfies the RIP of order $ 2s $ with $ \delta_{2s}=\hat{\delta} $ is at least
\begin{equation}
	\beta:= 1-2e^{-a_2 p}.
\end{equation}

If $ A(v_k) $ has independent entries following \eqref{distribution: Bernoulli} (or \eqref{distribution: Bernoulli-like}), then $ \|v_k^j\| $ is bounded by $ \kbv = \sqrt{n/p}$ (or $ \kbv = \sqrt{3n/p}$).
If $ A(v_k) $ has independent entries following \eqref{distribution: Gaussian}, then $ \|v_k^j\|$ follows a chi-square distribution and has an upper bound with high probability (cf. \cite{inglot2010inequalities}).

\section{Almost-sure global convergence}\label{sec: convergence}
In this section, we study convergence properties of Algorithm \ref{algo: sparseDFOLM}.
It is shown that Algorithm \ref{algo: sparseDFOLM} converges to a stationary point of the sparse nonlinear least squares problem with probability one.


Define the level set
$$
	\mathcal{L}_0:=\{x\in\R^n \mid \norm{F(x)}\leq \norm{F(x_0)}\}.
$$
By \eqref{algo-eq: update xk}, the sequence $\{x_k\}$ generated by Algorithm \ref{algo: sparseDFOLM} lies in
$\mathcal{L}_0$. Denote
\begin{equation*}
	\mathcal{L}:= \mathcal{L}_0 \;\bigcup\; \textstyle \bigcup\limits_{x\in\mathcal{L}_0} \! B(x,\theta_{\min}^{-1}).
\end{equation*}
We make the following assumptions.

\begin{assumption}\label{assp: L compact}
	(i) $ \mathcal{L} $ is compact.

	(ii) $ F(x) $ is Lipschitz continuous in $ \mathcal{L} $, i.e., there exist $ \klf>0 $ such that
	\begin{equation}\label{assp-eq: F Lipschitz}
		\norm{F(y)-F(x)}\leq \klf\norm{y-x},\quad \forall x,y\in \mathcal{L}.
	\end{equation}
\end{assumption}

By Assumption \ref{assp: L compact}, there exist $ \kbf$ such that
\begin{equation}\label{assp-eq: F bounded}
	\norm{F(x)}\leq \kbf,\quad \forall x\in \mathcal{L},
\end{equation}
and
\begin{equation}
	\norm{J(x)}\leq \klf,\quad \forall x\in \mathcal{L}.
	\label{assp-eq: J bounded}
\end{equation}

\begin{assumption}\label{assp: Jmk bounded}
	For every realization of Algorithm \ref{algo: sparseDFOLM}, $ \Jmk $ is bounded for all $ k $, i.e., there exists $ \kbmj > 0 $ such that
	\begin{equation}\label{assp-eq: Jmk bounded}
		\norm{\Jmk}\leq \kbmj,\quad \forall k.
	\end{equation}
\end{assumption}

It can be verified that $d_k$ is also a solution to the trust region subproblem
\begin{equation}\label{prob: trust-region}
	\min_{d\in\R^n}\ \frac{1}{2}\norm{F_k+\Jmk d}^2
	\quad \text{s.t.}\ \norm{d}\leq\norm{d_k}.
\end{equation}
By Powell's result \cite[Theorem 4]{powell1975ConvergencePropertiesClass}, we have
\begin{lemma}\label{lm: cauchy decrease}
	The predicted reduction satisfies
	\begin{equation}
		\norm{F_k}^2 - \norm{F_k+\Jmk d_k}^2 \geq \norm{\Jmk^T F_k}\min\left\{ \norm{d_k}, \frac{\norm{\Jmk^T F_k}}{\norm{\Jmk^T \Jmk}}\right\}
	\end{equation}
	for all $ k $.
\end{lemma}

We first prove that the $k$-th iteration is successful if $ \Jmk $ is first-order accurate and
$\theta_k\norm{\Jmk^T F_k}$ is larger than some positive constant.

\begin{lemma}\label{lm: M1 lambda large then success}
	Suppose that Assumptions \ref{assp: J Lipschitz}, \ref{assp: L compact} and \ref{assp: Jmk bounded} hold. Consider the $ k $-th iteration of a realization of Algorithm \ref{algo: sparseDFOLM}. If \eqref{def-eq: Jacobian accurate} is satisfied and
	\begin{equation}\label{thetak}
		\theta_k \geq \frac{\klj\kbf+\klf^2+2\kbf\kej+\kbmj^2}{(1-\eta_0)\norm{\Jmk^T F_k}}:=\frac{M}{\norm{\Jmk^T F_k}},
	\end{equation}
	then $ \rho_k\geq \eta_0 $.
\end{lemma}

\begin{proof}
	By \eqref{assp-eq: jacobian Lipschitz} and \eqref{assp-eq: F bounded}--\eqref{assp-eq: J bounded}, for any $ x,y\in\mathcal{L} $,
	\begin{align}
		& \norm{J(x)^TF(x)-J(y)^TF(y)} \nonumber \\
		\leq {} & \norm{J(x)^T(F(x)-F(y))} +\norm{(J(x)-J(y))^TF(y)}\nonumber \\
		\leq {} & (\klf^2+\klj\kbf)\norm{x-y}.\label{eq: JTF Lipschitz}
	\end{align}
	Hence the gradient of $ \frac12\norm{F(x)}^2 $ is Lipschitz continuous.
	So it holds
	\begin{equation}\label{eq: f2 quadratic bound}
		\begin{aligned}
			\big|\norm{F(y)}^2\!-\!\norm{F(x)}^2\!-\!(2J(x)^TF(x))^T(y-x)\big|
			\leq{} & (\klf^2+\klj\kbf)\norm{x\!-\!y}^2.
		\end{aligned}
	\end{equation}
	Combining \eqref{def-eq: Jacobian accurate}, \eqref{assp-eq: F bounded}, \eqref{assp-eq: Jmk bounded} and \eqref{eq: f2 quadratic bound}, we obtain
	\begin{align}	\label{eq: numerator decrease}
		& \left|\norm{F_k+\Jmk d_k}^2 \!-\! \norm{F(x_k+d_k)}^2\right|\nonumber \\
		\leq{} & \left|\norm{F_k}^2 + 2(J_k^T F_k)^Td_k- \norm{F(x_k+d_k)}^2\right|+ \left|2((\Jmk-J_k)^T F_k)^Td_k\right| + \norm{\Jmk d_k}^2\nonumber \\
		\leq {} & (\klj\kbf+\klf^2) \norm{d_k}^2 + 2\kbf\kej\theta_k^{-1}\norm{d_k}+\kbmj^2\norm{d_k}^2.
	\end{align}

	By \eqref{assp-eq: Jmk bounded}, \eqref{thetak} and the definition of $d_k$,
	\begin{align}
		\norm{d_k}\leq \frac{1}{\theta_k}
		\leq \frac{(1-\eta_0)\norm{\Jmk^T F_k}}{ \kbmj^2}
		\label{eq: dk upper bound 2}
		\leq 	\frac{\norm{\Jmk^T F_k}}{\norm{\Jmk^T \Jmk}}.
	\end{align}
	It then follows from Lemma \ref{lm: cauchy decrease} that
	\begin{equation}\label{eq: pred decrease}
		\mathrm{Pred}_k = \norm{F_k}^2 - \norm{F_k+\Jmk d_k}^2 \geq \norm{\Jmk^T F_k} \norm{d_k}.
	\end{equation}
	Hence, by \eqref{thetak}, \eqref{eq: numerator decrease}--\eqref{eq: pred decrease},
	\begin{align}
		|\rho_k-1| & =\left| \frac{\norm{F_k+\Jmk d_k}^2-\norm{F(x_k+d_k)}^2}{\norm{F_k}^2-\norm{F_k+\Jmk d_k}^2}\right| \nonumber\\
		& \leq \frac{(\klj\kbf+\klf^2) \norm{d_k}^2 + 2\kbf\kej\theta_k^{-1}\norm{d_k}+\kbmj^2\norm{d_k}^2}{\norm{\Jmk^T F_k} \norm{d_k}} \nonumber\\
		& \leq \frac{\klj\kbf+\klf^2+2\kbf\kej+\kbmj^2}{\theta_k \norm{\Jmk^T F_k}} \nonumber\\
		& \leq 1-\eta_0.
	\end{align}
	This implies that $ \rho_k \geq \eta_0 $. The proof is completed.
\end{proof}

Next we prove that the sequence $ \{\theta_k\} $ goes to infinity.

\begin{lemma}\label{lm: theta goes to infinity}
	Suppose that Assumptions \ref{assp: L compact} and \ref{assp: Jmk bounded} hold.
	For any realization of Algorithm \ref{algo: sparseDFOLM},
	\begin{equation}\label{lm-eq: theta goes to infinity}
		\lim_{k\rightarrow\infty} \theta_k = +\infty.
	\end{equation}
\end{lemma}

\begin{proof}
	We prove by contradiction.
	Suppose that \eqref{lm-eq: theta goes to infinity} does not hold.
	Then there exists $\tilde \theta>0$ such that the set $\Omega_1=\{k\mid \theta_k< \tilde
		\theta\}$ is infinite. Hence the set $\Omega_2=\{k\mid \theta_k< \frac{\tilde \theta}{\gamma_1}\}$
	is also infinite because of $0<\gamma_1<1$.

	Define the set
	\begin{align}\label{2.12}
		\Omega_3=\{k\in \Omega_2\mid \theta_{k+1}\leq\theta_k\}.
	\end{align}
	We show that $ \Omega_3 $ is infinite too.
	Otherwise, there exists $K_0\in \Omega_2$ such that $\theta_{k+1}>\theta_k$ for all $k\geq K_0$ and $k\in \Omega_2$.
	Then there exists $K_1>K_0$ such that $\theta_{K_1}\geq 	\frac{\tilde \theta}{\gamma_1}$ and $\theta_{k+1}>\theta_k$ for all $K_0\leq k< K_1$.
	Since $\Omega_2$ is infinite, by \eqref{algo-eq: update LM para},
	there exists $K_2>K_1$ such that $\tilde \theta\leq\theta_{K_2}<\frac{\tilde \theta}{\gamma_1}$ and
	$\theta_k\geq \frac{\tilde \theta}{\gamma_1}$ for all $K_1\leq k< K_2$.
	By induction, we obtain $\theta_k\geq \tilde \theta$ for all $k>K_2$.
	This is a contradiction to the infiniteness of $\Omega_1$. So $\Omega_3$ is infinite.


	By \eqref{algo-eq: update LM para}, for all $ k\in\Omega_3 $,
	\begin{equation}\label{eq: JmkF geq eta_1/c}
		\norm{\Jmk^TF_k}\geq \frac{\eta_1}{\theta_k} \geq \frac{\eta_1\gamma_1}{\tilde \theta}.
	\end{equation}
	This, together with \eqref{assp-eq: Jmk bounded}, implies
	\begin{equation}
		\norm{d_k} \geq \frac{\norm{\Jmk^TF_k}}{\norm{\Jmk^T\Jmk}+\theta_k\norm{\Jmk^TF_k}} \geq \frac{\eta_1\gamma_1}{\tilde\theta (\kbmj^2+ \eta_1)}.
	\end{equation}
	It then follows from Lemma \ref{lm: cauchy decrease} and $ \rho_k\geq\eta_0$ that for any $k\in\Omega_3$,
	\begin{align}
		\norm{F_k}^2 - \norm{F(x_k+d_k)}^2 & \geq \eta_0 (\norm{F_k}^2-\norm{F_k+\Jmk d_k}^2) \nonumber\\
		& \geq \eta_0\norm{\Jmk^T F_k}\min\left\{\norm{d_k},\frac{\norm{\Jmk^T F_k}}{\norm{\Jmk^T \Jmk}}\right\} \nonumber\\
		& \geq \frac{\eta_0\eta_1^2\gamma_1^2}{\tilde \theta^2 (\kbmj^2+ \eta_1)}.
	\end{align}
	This contradicts the infiniteness of $ \Omega_3 $ and $ \norm{F_k} \geq 0$. So \eqref{lm-eq: theta goes to infinity} holds true. The proof is completed.
\end{proof}

We are now ready to prove that at least one accumulation point of the sequence generated by Algorithm \ref{algo: sparseDFOLM} is a stationary point of the sparse nonlinear least squares problem \eqref{eq: nonlinear least squares problems} with probability one.

\begin{theorem}\label{thm: liminf convergence}
	Suppose that Assumptions \ref{assp: J Lipschitz}, \ref{assp: L compact} and \ref{assp: Jmk bounded} hold and $ \beta\geq \frac{\log\gamma_2}{\log\gamma_2-\log\gamma_1} $.
	Let $\sigma_k=\norm{d_{k-1}}$ and $\xi_k =\sqrt{p}\klg \kbv^2\sigma_k/2 $.
	If $\{A(V_k)\} $ is a sequence of $ \beta $-probabilistically $ \kbv $-bounded RIP matrices,
	then the random iterates $ \{X_k\} $ generated by Algorithm \ref{algo: sparseDFOLM} satisfies
	\begin{equation}\label{thm-eq: liminf convergence}
		\P\left(\liminf_{k\rightarrow\infty}\norm{J^T(X_k)F(X_k)}=0\right)=1.
	\end{equation}
\end{theorem}

\begin{proof}
	Consider the random walk
	\begin{equation}
		W_j = \sum_{k=0}^{j}\left(\ITk -\beta\right),
	\end{equation}
	where $ \ITk $ is the indicator function of the event $ T_k $. Then
	\begin{equation}
		W_j = \left\{\begin{array}{ll}
			W_{j-1}+1-\beta, & \text{if}\ \mathbf{1}_{T_j}=1, \\
			W_{j-1}-\beta, & \text{if}\ \mathbf{1}_{T_j}=0. \\
		\end{array}\right.
	\end{equation}
	By \eqref{def-eq: Tk submartingale},
	\begin{align}
		\E\left[W_j\mid \mathcal{F}_{j-1}\right] & = \E\left[W_{j-1}\mid \mathcal{F}_{j-1}\right] +\E\left[\mathbf{1}_{T_j} -\beta\mid \mathcal{F}_{j-1}\right]\nonumber \\
		& = W_{j-1} + \P\left(T_j\mid\mathcal{F}_{j-1}\right) -\beta \nonumber\\
		& \geq W_{j-1}.
	\end{align}
	Hence $ W_j $ is a submartingale with bounded increments.
	Thus, according to \cite[Theorem 4.1]{bandeira2014ConvergenceTrustRegionMethods}, we obtain
	\begin{equation}
		\P\left(\Big\{\lim_{j\rightarrow\infty}W_j \text{ exists and is finite}\Big\} \, {\textstyle \bigcup}\, \Big\{\limsup_{j\rightarrow\infty}W_j = +\infty \Big\} \right) = 1.
	\end{equation}
	Since $ |W_{j+1}-W_j |\geq \min\{\beta, 1-\beta\}$, the sequence $ \{W_j\} $ has no limit, so we have
	\begin{equation}\label{eq: Wj infty}
		\P\left( \limsup_{j\rightarrow\infty}W_j = +\infty \right) = 1.
	\end{equation}

	Now we prove \eqref{thm-eq: liminf convergence} by contradiction. Without loss of generality, suppose that there exists $ \varepsilon>0 $ such that
	\begin{equation}\label{eq: liminf convergence contradiction assumption}
		\norm{J^T(X_k)F(X_k)} \geq \varepsilon
	\end{equation}
	holds for all $ k $ with positive probability.
	By Lemma \ref{lm: theta goes to infinity}, for any realization of Algorithm \ref{algo: sparseDFOLM}, there exists $ K >1$ such that for all $ k\geq K $,
	\begin{equation}\label{eq: liminf convergence theta}
		\theta_k\geq \max\left\{\frac{2\kej\kbf}{\varepsilon}, \frac{2M}{\varepsilon}, \frac{2\eta_2}{\varepsilon}, \frac{\theta_{\min}}{\gamma_1}\right\},
	\end{equation}
	where $ M $ is a constant given in \eqref{thetak}.
	We discuss the case when $ \norm{J^T(X_k)F(X_k)} \geq \varepsilon $ happens. Define the random variable $ Z_k $ with the realization
	\begin{equation}
		z_k = (1-\beta)\log_{\gamma_1}(\theta_k).
	\end{equation}
	In the following, we consider two cases: $ \ITk=1 $ and $ \ITk=0 $.

	(i) $ \ITk=1 $. By Theorem \ref{thm: Jacobian accurate}, we have
	$ \displaystyle \norm{\Jmk-J_k}\leq \frac{\kej}{\theta_k} $. This, together with \eqref{eq: liminf convergence theta}, yields
	\begin{equation}\label{eq: liminf convergence JmF}
		\norm{\Jmk^T F_k} \geq \norm{J_k^T F_k} -\norm{\Jmk^T F_k - J_k^T F_k} \geq \varepsilon - \frac{\kej\kbf}{\theta_k} \geq \frac{\varepsilon}{2}.
	\end{equation}
	Combining \eqref{eq: liminf convergence theta} with \eqref{eq: liminf convergence JmF}, we obtain
	\begin{equation}
		\theta_k \geq \frac{2M}{\varepsilon} \geq \frac{M}{\norm{\Jmk^T F_k}}.
	\end{equation}
	It then follows from Lemma \ref{lm: M1 lambda large then success} that $ \rho_k \geq \eta_0 $. Since
	\begin{equation}
		\norm{\Jmk^T F_k}\geq \frac{\varepsilon}{2} \geq \frac{\eta_2}{\theta_k}, \ \gamma_1\theta_k\geq\theta_{\min},
	\end{equation}
	by the update rule of $ \theta_k $, we have $ \theta_{k+1} = \gamma_1\theta_k $.
	Thus
	\begin{equation}\label{eq: zk wk a}
		z_{k+1} - z_k = (1-\beta)\left(\log_{\gamma_1}(\gamma_1\theta_k) - \log_{\gamma_1}(\theta_k) \right)= 1-\beta = w_k - w_{k-1},
	\end{equation}
	where $ w_k $ is a realization of $ W_k $.

	(ii) $ \ITk =0 $. Note that $ \theta_{k+1}\leq \gamma_2\theta_k $ always holds. Hence
	\begin{equation}
		z_{k+1} - z_k \geq (1-\beta)\left(\log_{\gamma_1}(\gamma_2\theta_k) - \log_{\gamma_1}(\theta_k) \right)= (1-\beta)\log_{\gamma_1}(\gamma_2).
	\end{equation}
	Since $ \beta\geq \frac{\log\gamma_2}{\log\gamma_2-\log\gamma_1} $, we have
	\begin{equation}\label{eq: zk wk b}
		z_{k+1} - z_k \geq -\beta = w_{k} - w_{k-1}.
	\end{equation}
	It then follows from \eqref{eq: zk wk a} and \eqref{eq: zk wk b} that
	\begin{equation}\label{eq: zk geq wk}
		z_{k+1} - z_{K} \geq w_k-w_{K-1}, \ \forall k>K.
	\end{equation}
	Since $ z_{K} $ and $ w_{K-1} $ are finite, \eqref{eq: Wj infty} and \eqref{eq: zk geq wk} imply that
	\begin{equation}\label{eq: Zj infty}
		\P\left( \limsup_{j\rightarrow\infty}Z_j = +\infty \right) = 1,
	\end{equation}
	which contradicts $ \theta_k \rightarrow +\infty $ and $ \gamma_1<1 $. So \eqref{eq: liminf convergence contradiction assumption} cannot be true. The proof is completed.
\end{proof}


\begin{lemma}\label{lm: sum reciprocal theta finite}
	Suppose that Assumptions \ref{assp: J Lipschitz}, \ref{assp: L compact} and \ref{assp: Jmk bounded} hold and $ \beta\geq \frac{\log\gamma_2}{\log\gamma_2-\log\gamma_1} $.
	Let $\sigma_k=\norm{d_{k-1}}$ and $\xi_k =\sqrt{p}\klg \kbv^2\sigma_k/2 $.
	For any given $ \varepsilon>0 $, define
	\begin{equation}
		\Omega =\{k \mid \norm{J^T(X_k)F(X_k)}>\varepsilon\}.
	\end{equation}
	If $\{A(V_k)\} $ is a sequence of $ \beta $-probabilistically $ \kbv $-bounded RIP matrices, then the sequence $ \{\Theta_k\} $ generated by Algorithm \ref{algo: sparseDFOLM} satisfies
	\begin{equation}
		\P\left(\sum_{k\in\Omega}\frac{1}{\Theta_k}<+\infty\right)=1.
	\end{equation}
\end{lemma}

\begin{proof}
	Define
	\begin{equation}
		\Omega_1 = \{k\in \Omega \mid T_k \text{ happens} \} \text{ and } \Omega_2 = \Omega\backslash \Omega_1.
	\end{equation}
	It follows from Theorem \ref{thm: Jacobian accurate} that
	$ \displaystyle \norm{\Jmk-J_k}\leq \frac{\kej}{\theta_k} $ for all $ k\in\Omega_1 $. We first prove that
	\begin{equation}\label{eq: theta omega 1}
		\sum_{k\in\Omega_1}\frac{1}{\theta_k}<+\infty.
	\end{equation}
	Obviously, \eqref{eq: theta omega 1} holds true when $ \Omega_1 $ is finite.
	When $ \Omega_1 $ is infinite, by a similar analysis as that in Theorem \ref{thm: liminf convergence}, we have
	\begin{equation}\label{eq: JmkFk geq epsilon/2}
		\norm{\Jmk^T F_k}\geq\frac{\varepsilon}{2} \text{ and } \rho_k \geq \eta_0
	\end{equation}
	for sufficiently large $ k\in\Omega_1 $.
	It then follows from \eqref{assp-eq: F bounded}, \eqref{assp-eq: Jmk bounded}, \eqref{eq: JmkFk geq epsilon/2} and Lemma \ref{lm: cauchy decrease} that for sufficiently large $ k\in\Omega_1 $,
	\begin{align}
		\norm{F_k}^2 - \norm{F(x_k+d_k)}^2 & \geq \eta_0 (\norm{F_k}^2-\norm{F_k+\Jmk d_k}^2) \nonumber\\
		& \geq \eta_0\norm{\Jmk^T F_k}\min\left\{\norm{d_k},\frac{\norm{\Jmk^T F_k}}{\norm{\Jmk^T \Jmk}}\right\}\nonumber \\
		& \geq \eta_0\frac{\norm{\Jmk^T F_k}^2}{\norm{\Jmk^T \Jmk}+\theta_k\norm{\Jmk^T F_k}} \nonumber\\
		& \geq \frac{\eta_0\theta_k}{\kbmj^2+\theta_k\kbmj\kbf} \frac{\norm{\Jmk^T F_k}^2}{\theta_k} \nonumber\\
		& \geq \frac{\eta_0\theta_{\min}\varepsilon^2}{4(\kbmj^2+\theta_{\min}\kbmj\kbf)}\frac{1}{\theta_k}.
	\end{align}
	Thus \eqref{eq: theta omega 1} holds true.

	Define the random variable $G_0 = 0$ and $G_k = G_{k-1} + (\ITk - \beta)\Theta_k^{-1} ~(k \geq 1) $. Since $\{A(V_k)\} $ is a sequence of $ \beta $-probabilistically $ \kbv $-bounded RIP matrices,
	\begin{align}
		\E\left[G_k\mid \mathcal{F}_{k-1}\right] & = \E\left[G_{k-1}\mid \mathcal{F}_{k-1}\right] + \E\left[(\ITk-\beta)\Theta_k^{-1}\mid \mathcal{F}_{k-1}\right] \nonumber\\
		& = G_{k-1} + \Theta_k^{-1}\E\left[\ITk-\beta\mid \mathcal{F}_{k-1}\right] \nonumber \\
		& \geq G_{k-1}.
	\end{align}
	Thus $ G_k $ is a submartingale with bounded increments. It then follows from \cite[Theorem 4.1]{bandeira2014ConvergenceTrustRegionMethods} that
	\begin{equation}\label{eq: G_k limit}
		\P\left(\Big\{\lim_{k\rightarrow\infty}G_k \text{ exists and is finite}\Big\} \,{\textstyle \bigcup}\, \Big\{\limsup_{k\rightarrow\infty}G_k = +\infty \Big\} \right) = 1.
	\end{equation}
	Note that
	\begin{equation}
		G_k=(1-\beta)\sum_{k\in\Omega_1}\frac{1}{\Theta_k}-\beta\sum_{k\in\Omega_2}\frac{1}{\Theta_k}.
	\end{equation}
	If the event
	$$
		\left\{\sum_{k\in\Omega_2}\frac{1}{\Theta_k} = +\infty \right\}
	$$
	happens with some positive probability, then by \eqref{eq: theta omega 1},
	the event
	\begin{equation}\label{eq: event 2}
		\Big\{\limsup_{k\rightarrow \infty} G_k = -\infty\Big\}
	\end{equation}
	also happens with some positive probability, which is a contradiction to \eqref{eq: G_k limit}.
	Thus
	\begin{equation}
		\P\left(\sum_{k\in\Omega_2}\frac{1}{\Theta_k}<+\infty\right) =1
	\end{equation}
	and
	\begin{equation}
		\sum_{k\in\Omega}\frac{1}{\Theta_k} = \sum_{k\in\Omega_1}\frac{1}{\Theta_k}+\sum_{k\in\Omega_2}\frac{1}{\Theta_k}<+\infty
	\end{equation}
	holds almost surely.
\end{proof}

Now we prove that the sequence generated by Algorithm \ref{algo: sparseDFOLM} converges to a first-order stationary point of the problem \eqref{eq: nonlinear least squares problems} almost surely.

\begin{theorem}\label{thm: lim convergence}

	Suppose that Assumptions \ref{assp: J Lipschitz}, \ref{assp: L compact} and \ref{assp: Jmk bounded} hold and $ \beta\geq \frac{\log\gamma_2}{\log\gamma_2-\log\gamma_1} $.
	Let $\sigma_k=\norm{d_{k-1}}$ and $\xi_k =\sqrt{p}\klg \kbv^2\sigma_k/2 $.
	If $\{A(V_k)\} $ is a sequence of $ \beta $-probabilistically $ \kbv $-bounded RIP matrices,
	then the random iterates $ \{X_k\} $ generated by Algorithm \ref{algo: sparseDFOLM} satisfies
	\begin{equation}\label{thm-eq: lim convergence}
		\P\left(\lim_{k\rightarrow\infty}\norm{J^T(X_k)F(X_k)}=0\right)=1.
	\end{equation}
\end{theorem}

\begin{proof}
	We prove by contradiction. Suppose that \eqref{thm-eq: lim convergence} is not true.
	Then there exists $ \varepsilon>0 $ such that $ \Omega_1 =\{k\mid \norm{J^T_k F_k}>2\varepsilon\} $ is infinite holds with positive probability. Define
	\begin{equation}
		\Omega_2 =\{k\mid \norm{J^T_k F_k}>\varepsilon\}.
	\end{equation}
	By Theorem \ref{thm: liminf convergence}, there exist infinite integer pairs $ (k', k'') $ with $ 0<k'<k'' $ such that
	\begin{equation}
		\norm{J^T_{k'}F_{k'}}\leq\varepsilon,\ \norm{J^T_{k'+1}F_{k'+1}}>\varepsilon, \ \norm{J^T_{k''}F_{k''}}> 2\varepsilon
	\end{equation}
	and for any $ k $ such that $ k'< k < k'' $,
	\begin{equation}
		\varepsilon<\norm{J^T_k F_k} \leq 2\varepsilon
	\end{equation}
	hold with positive probability.
	Hence by \eqref{eq: JTF Lipschitz},
	\begin{align}
		\varepsilon < & \Big| \norm{J^T_{k''}F_{k''}} -\norm{J^T_{k'}F_{k'}} \Big| \nonumber\\
		\leq & \sum_{l=k'}^{k''-1} \Big| \norm{J^T_{l+1}F_{l}} -\norm{J^T_{l}F_{l}} \Big| \nonumber\\
		\leq & \sum_{l=k'}^{k''-1} (\klf^2+\klj\kbf)\norm{d_l} \nonumber\\
		\leq & (\klf^2+\klj\kbf)\sum_{l=k'}^{k''-1} \frac{1}{\theta_l}\nonumber\\
		= & (\klf^2+\klj\kbf)\left(\frac{1}{\theta_{k'}}+\sum_{l=k'+1}^{k''-1} \frac{1}{\theta_l}\right).
	\end{align}
	It then follows from Lemma \ref{lm: theta goes to infinity} that $ \frac{1}{\theta_{k'}}<\frac{\varepsilon}{2(\klf^2+\klj\kbf)} $ holds for sufficiently large $ k' $.
	Thus
	$ \sum_{l=k'+1}^{k''-1} \frac{1}{\theta_l} >\frac{\varepsilon}{2(\klf^2+\klj\kbf)} >0$.
	Therefore we have
	\begin{equation}
		\sum_{l\in\Omega_2} \frac{1}{\theta_l} = +\infty,
	\end{equation}
	which contradicts Lemma \ref{lm: sum reciprocal theta finite}.
	So \eqref{thm-eq: lim convergence} holds true.
	The proof is completed.
\end{proof}

\section{Numerical experiments}\label{sec: numerical results}
In this section, we test Algorithm \ref{algo: sparseDFOLM} (DFLM-SNLS) and compare it with four MATLAB's built-in optimization algorithms: \verb|fminunc| with the BFGS method (fminunc-BFGS), and \verb|lsqnonlin| with the LM method (lsqnonlin-LM), the interior point method (lsqnonlin-IP), and the trust region reflective method (lsqnonlin-TR), where the gradients are approximated by the forward finite differences with a step size of about $ 1.5\times 10^{-8} $. The experiments are implemented on a laptop with an Intel Core i9-10920X CPU (3.5 GHz) and 32GB of RAM, using MATLAB R2024a.

In our experiments, the sparse Jacobian model is constructed by solving the noiseless version of \eqref{prob: l1-Jacobian-noisy}, i.e., for $i=1,\ldots, m$, the $\ell_1$-minimization problem
\begin{equation}\label{prob: l1-Jacobian}
	\begin{aligned}
		\min_{g\in\R^n}\ & \norm{g}_1 \\
		\text{s.t.}\ & A(v_k)g = \sigma_k^{-1}y_k^{(i)}
	\end{aligned}
\end{equation}
is solved to obtain $g_k^{(i)}$,
where $A(v_k)$ is generated from one of the distributions \eqref{distribution: Gaussian}--\eqref{distribution: Bernoulli-like},
then $\Jmk$ is constructed by \eqref{jacobianmodel}.
We reformulate \eqref{prob: l1-Jacobian} as a linear programming problem and use the solver Gurobi 11.0.1 to find a minimizer.

In DFLM-SNLS, we set $\eta_0 = 10^{-3}$, $ \eta_1 = 10^{-4}$, $ \eta_2 = 10^{3}$, $ \gamma_1 = 0.25$, $ \gamma_2 = 4$, $ \theta_0 = 10^{-8}$, $ \theta_{\min} = 10^{-8}$, $ \sigma_0 = 1 $ and $ \sigma_k = \max\{10^{-9},\min\{10^{-7},\norm{d_{k-1}}\}\} $ for all tests.
The algorithm is terminated when
\[ \norm{\Jmk^TF_k}\leq 10^{-6},\quad\mbox{or}\quad \norm{d_k}\leq10^{-6},\quad\mbox{or}\quad \frac{\abs{\norm{F_k}^2-\norm{F(x_k+d_k)}^2}}{\norm{F_k}^2+10^{-8}}\leq10^{-6},\]
or the function evaluations exceeds $ 1000(n + 1) $.
For each problem, DFLM-SNLS runs 50 times and gives the average objective function value.
The results for the three distributions seem quite similar,
so we just present those where $ A(v_k) $ is generated from the Bernoulli distribution \eqref{distribution: Bernoulli}.

\begin{example}[\cite{Broyden1965ACO}] \label{ex:1}
	Consider the Broyden tridiagonal function
	\[F^{(i)}(x)=(3-2x_i)x_i -x_{i-1} -2x_{i+1} +1,\quad i=1,\ldots,n,\]
	where $x_0=x_{n+1}=0$. The initial point is $ (-1,\ldots,-1)^T $.
\end{example}

We observe that the Jacobian matrix of this problem exhibits a higher degree of sparsity with the growth in dimension.
So we test $ p\in\left\{\lceil n/2\rceil,~\lceil n/3\rceil,~\lceil n/4\rceil\right\} $ for $ n = 100 $ and $ p\in\left\{\lceil n/6\rceil,~\lceil n/8\rceil,~\lceil n/10\rceil\right\} $ for $ n= 500 $, respectively.
The results are presented in Figure \ref{fig: More1981-30}.
It seems that DFLM-SNLS performs much better than fminunc-BFGS, lsqnonlin-LM, lsqnonlin-TR and lsqnonlin-IP, while lsqnonlin-LM performs almost the same as lsqnonlin-TR, and fminunc-BFGS does not work very well.

\begin{figure}[htbp]
	\centering
	\subfigure[$n=100$]{
		\includegraphics[width=0.45\linewidth]{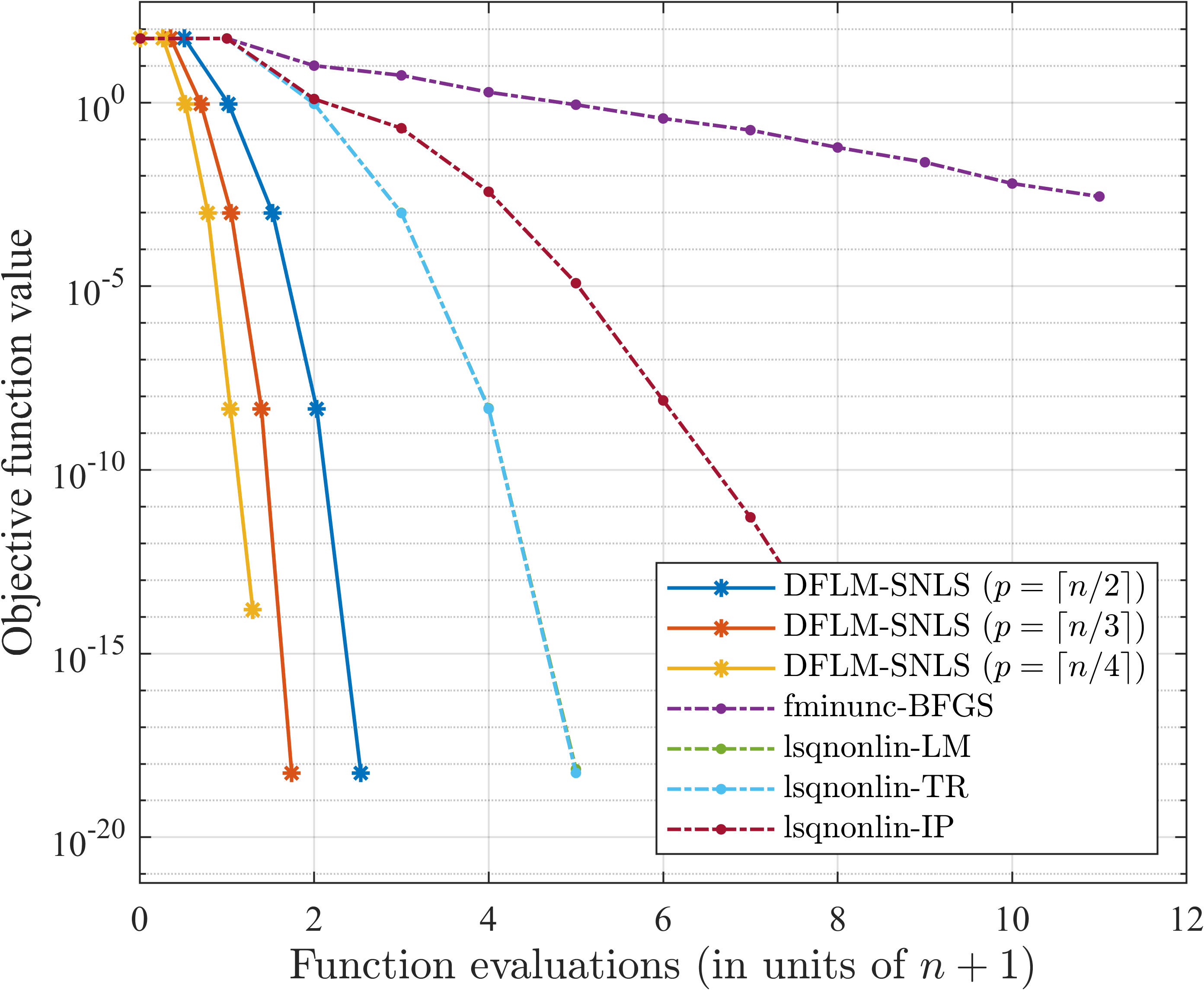}}\quad
	\subfigure[$n=500$]{
		\includegraphics[width=0.45\linewidth]{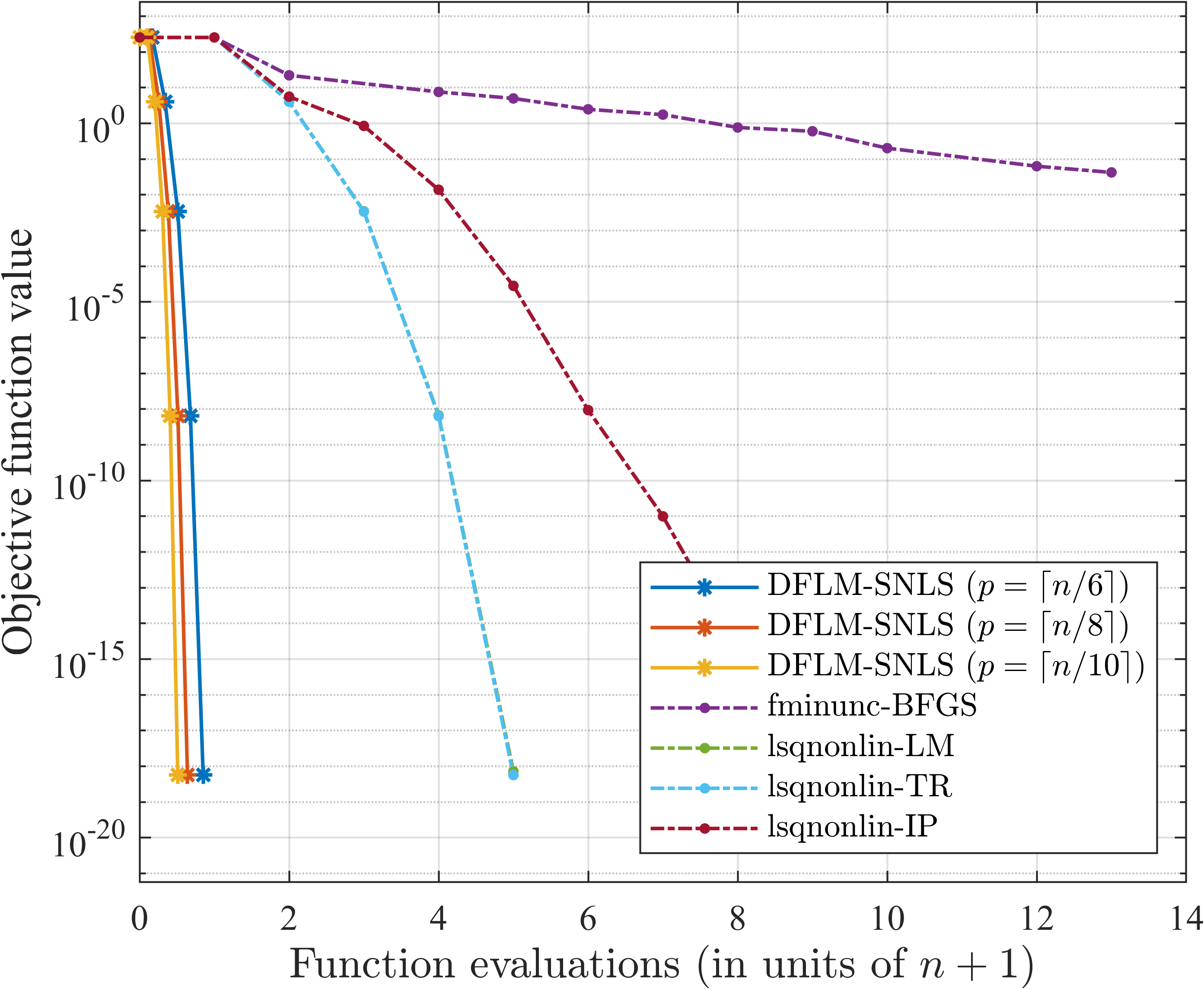}}
	\caption{Results for Example \ref{ex:1} with $n=100$ and $n=500$.}\label{fig: More1981-30}
\end{figure}

\begin{example}[\cite{friedlander1997SolvingNonlinearSystems}]\label{ex:2}
	Consider the tridimensional valley function, for $ i=1,\ldots,n/3 $ ($ n $ is a multiple of 3),
	\[ \begin{aligned}
			F^{(3i-2)}(x) & = (c_2x_{3i-2}^3+c_1x_{3i-2})\exp(-x_{3i-2}^2/100)-1, \\
			F^{(3i-1)}(x) & = 10(\sin(x_{3i-2})-x_{3i-1}), \\
			F^{(3i)}(x) & =10(\cos(x_{3i-2})-x_{3i}),
		\end{aligned} \]
	where $ c_1=1.003344481605351$ and $c_2=-3.344481605351171\times 10^{-3} $. The initial point is $ (-4,1,2,-4,1,2,\ldots)^T $.
\end{example}

We test the algorithms on the problem with $n=102$ and $n=501$.
The results are given in Figure \ref{fig: Friedlander1997-D5}.
DFLM-SNLS also performs best among the algorithms.

\begin{figure}[htbp]
	\centering
	\subfigure[$n=102$]{
		\includegraphics[width=0.45\linewidth]{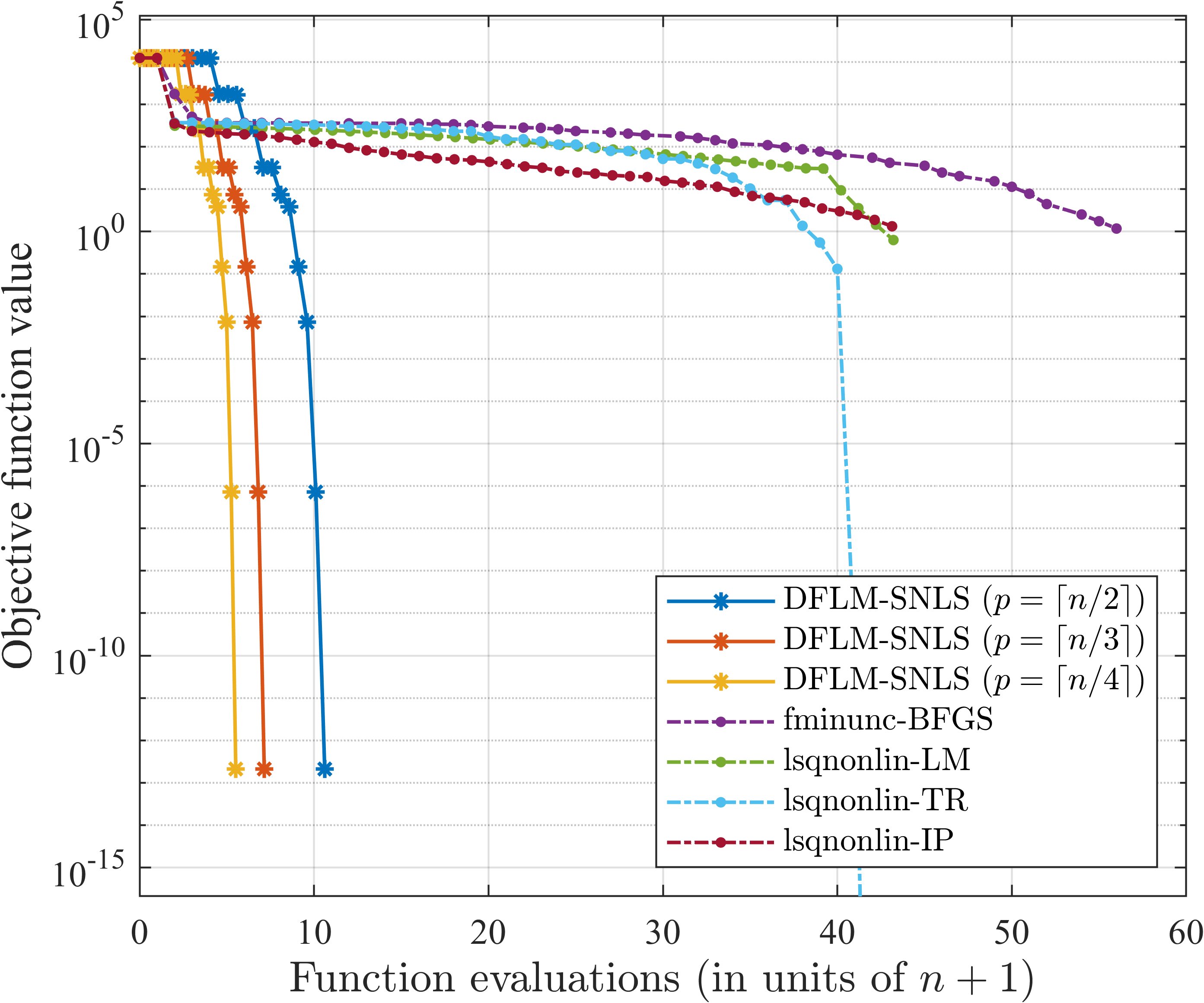}}\quad
	\subfigure[$n=501$]{
		\includegraphics[width=0.45\linewidth]{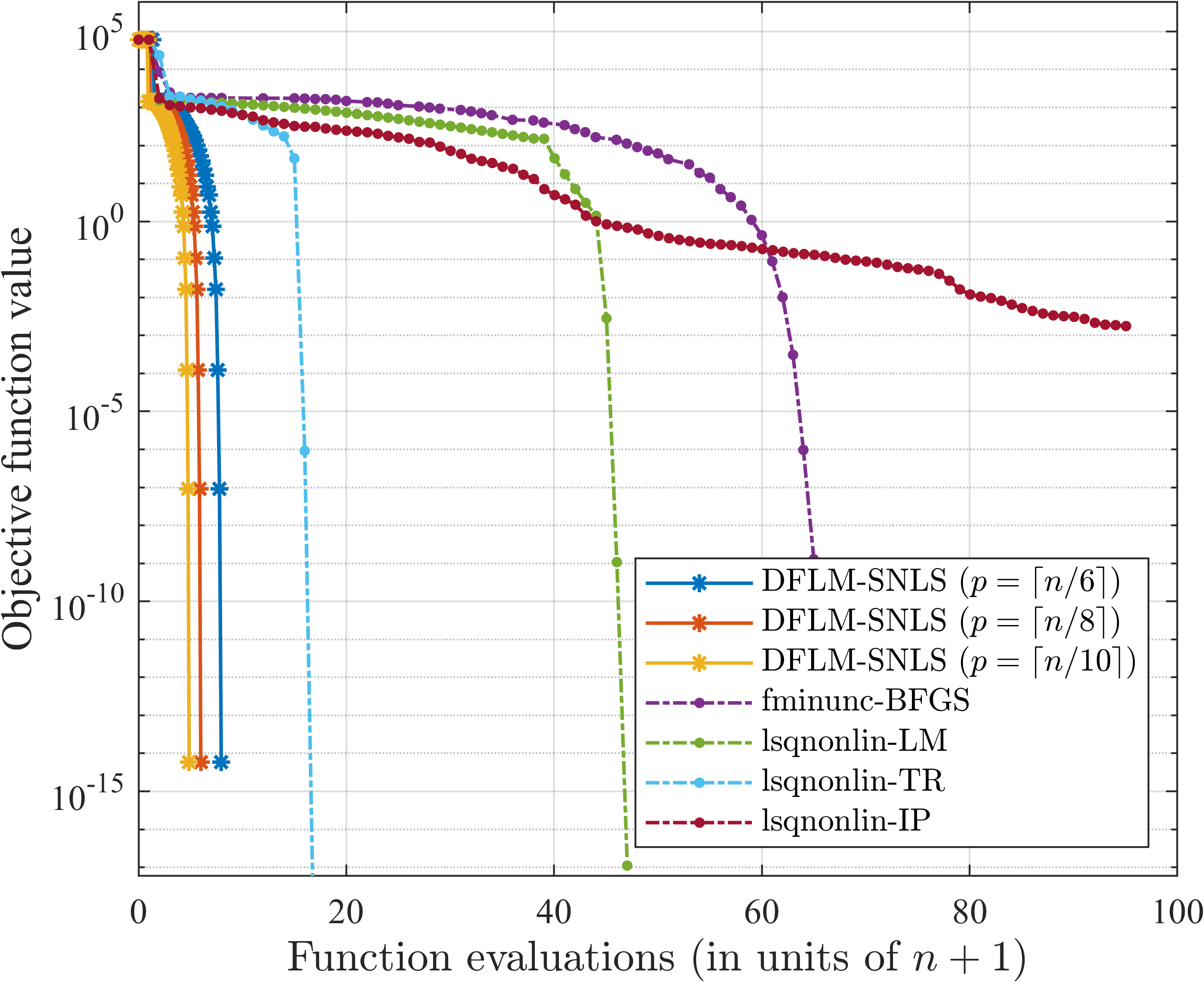}}
	\caption{Results for Example \ref{ex:2} with $n=102$ and $n=501$.}
	\label{fig: Friedlander1997-D5}
\end{figure}

\begin{example}[\cite{Yang1991AnEI}]\label{ex:3}
	Consider the extended Freudenstein and Roth function, for $ i=1,\ldots,n $,
	\[\begin{aligned}
			F^{(i)}(x) & = x_i+((5-x_{i+1})x_{i+1}-2)x_{i+1} -13, \quad \mathrm{mod}(i,2)=1, \\
			F^{(i)}(x) & = x_{i-1}+((x_i+1)x_i-14)x_i -29, \quad \mathrm{mod}(i,2)=0,
		\end{aligned}\]
	where $ \mathrm{mod}(a,b) $ is the remainder of $ a $ divided by $ b $. The initial point is $ (90,60,90,60,\ldots)^T $.
\end{example}

We test the algorithms on the problem with $n=100$ and $n=500$.
The results are given in Figure \ref{fig: Luksan2018-30}.

\begin{figure}[htbp]
	\centering
	\subfigure[$n=100$]{
		\includegraphics[width=0.45\linewidth]{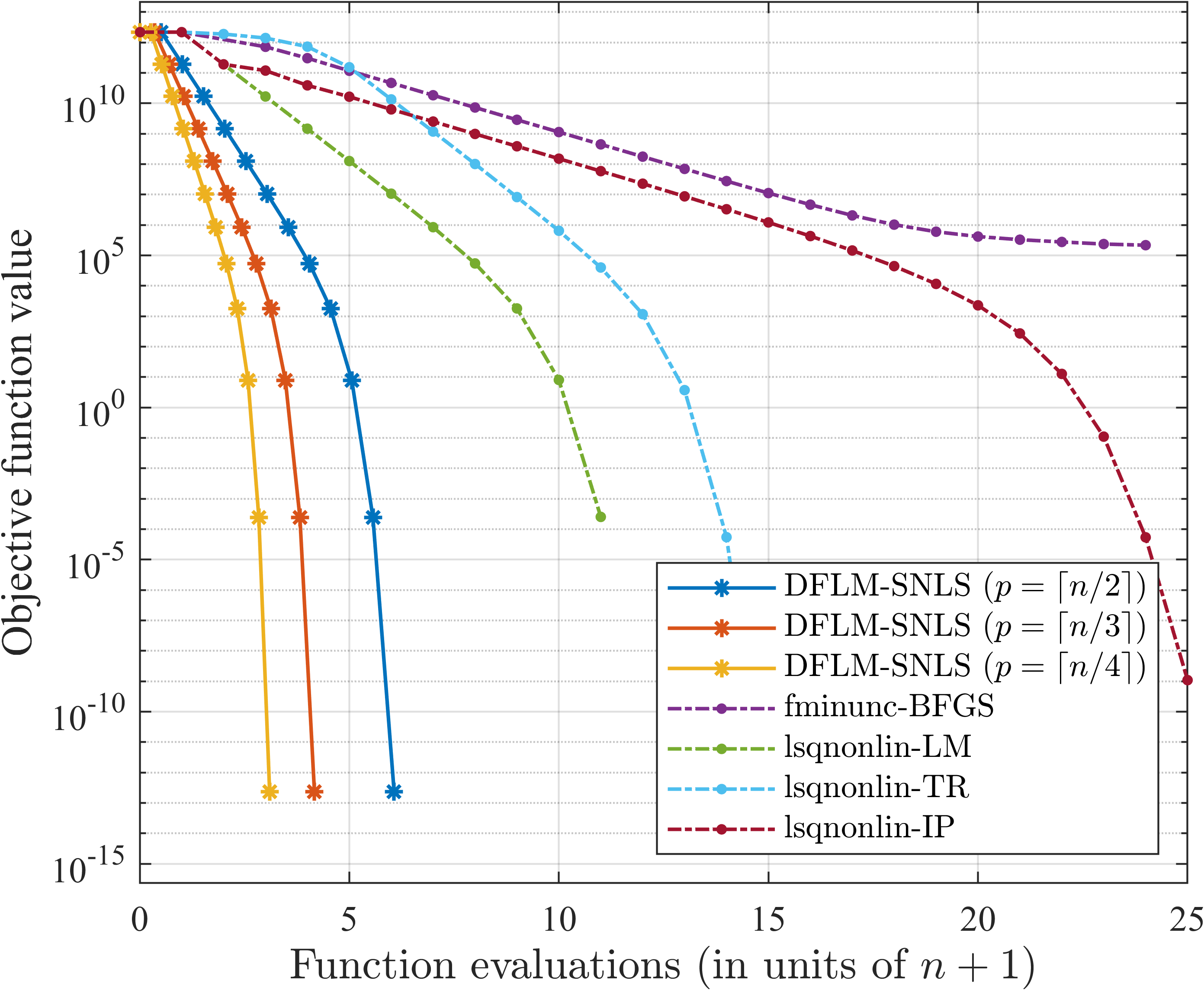}}\quad
	\subfigure[$n=500$]{
		\includegraphics[width=0.45\linewidth]{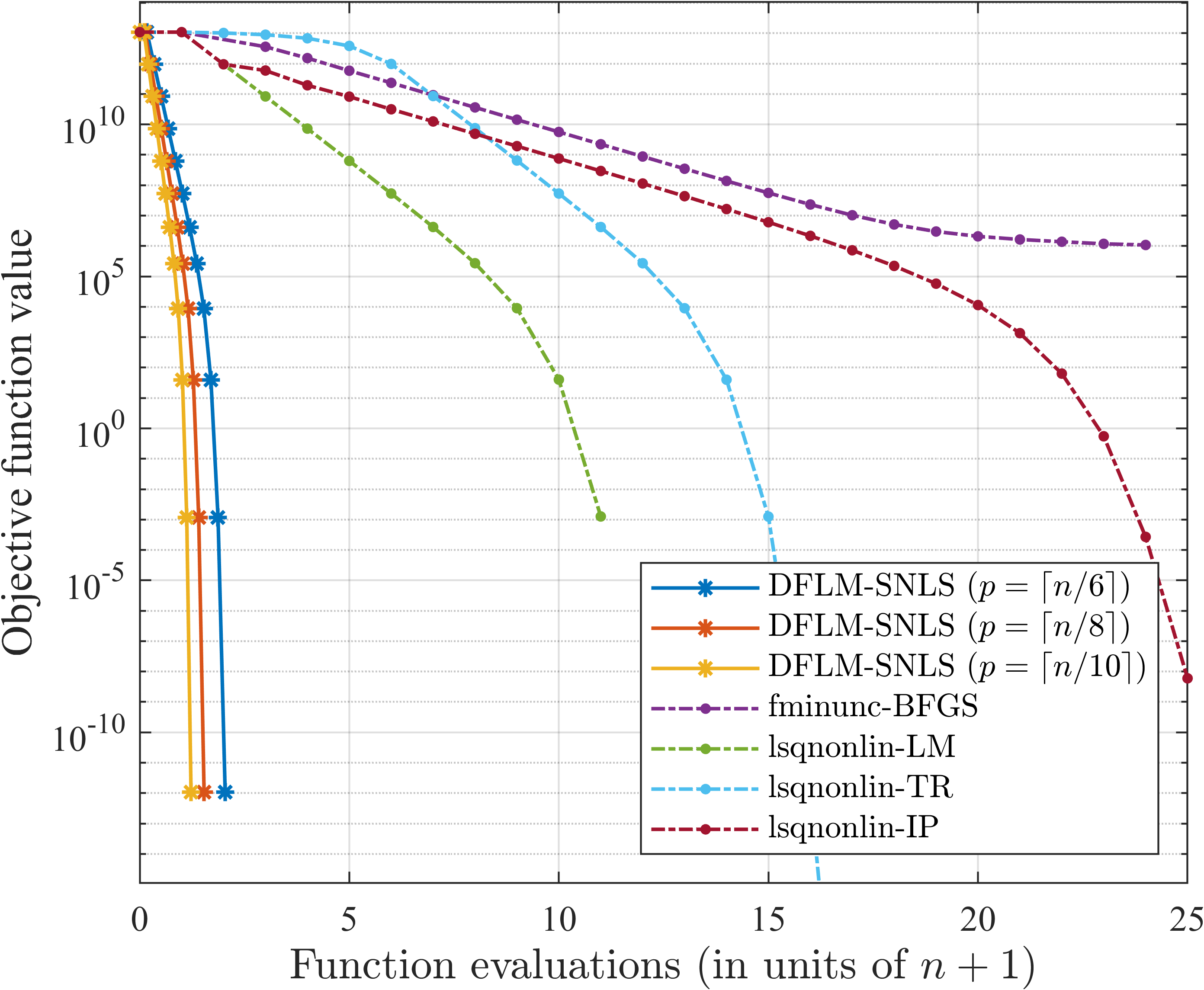}}
	\caption{Results for Example \ref{ex:3} with $n=100$ and $n=500$.}
	\label{fig: Luksan2018-30}
\end{figure}

\begin{example}[\cite{Toint1986NumericalSO}]\label{ex:4}
	Consider the trigonometric system, for $ i=1,\ldots,n $,
	\[ F^{(i)}(x) = 5-(l+1)(1-\cos(x_i))-\sin(x_i)-\sum_{j=5l+1}^{5l+5}\cos(x_j),~l=\lfloor \frac{i-1}{5} \rfloor. \]
	The initial point is $ (1/n,2/n,\ldots,n/n)^T $.
\end{example}

We test the algorithms on the problem with $n=100$ and $n=500$.
The results are given in Figure \ref{fig: Luksan2018-24}.

\begin{figure}[htbp]
	\centering
	\subfigure[$n=100$]{
		\includegraphics[width=0.45\linewidth]{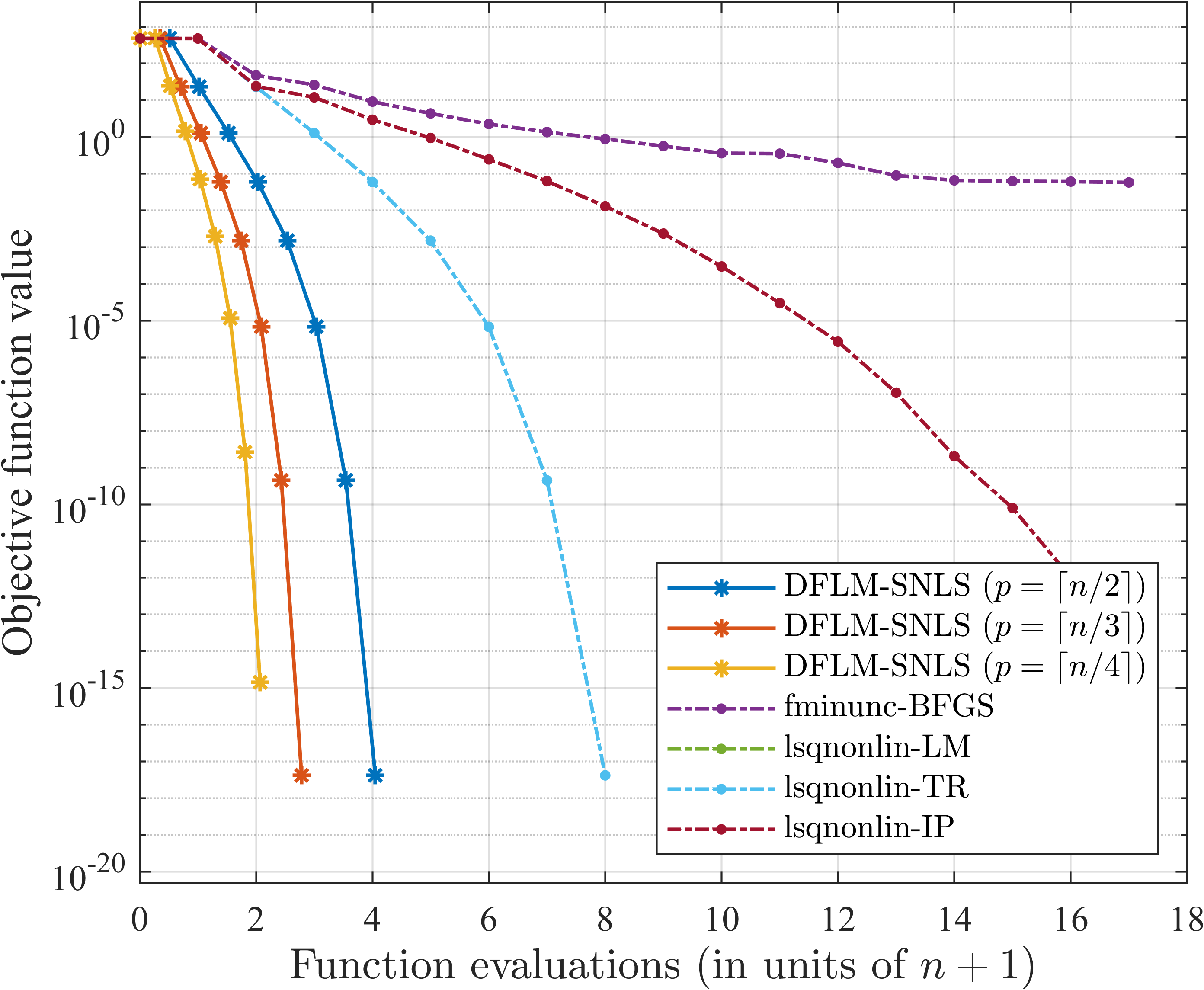}}\quad
	\subfigure[$n=500$]{
		\includegraphics[width=0.45\linewidth]{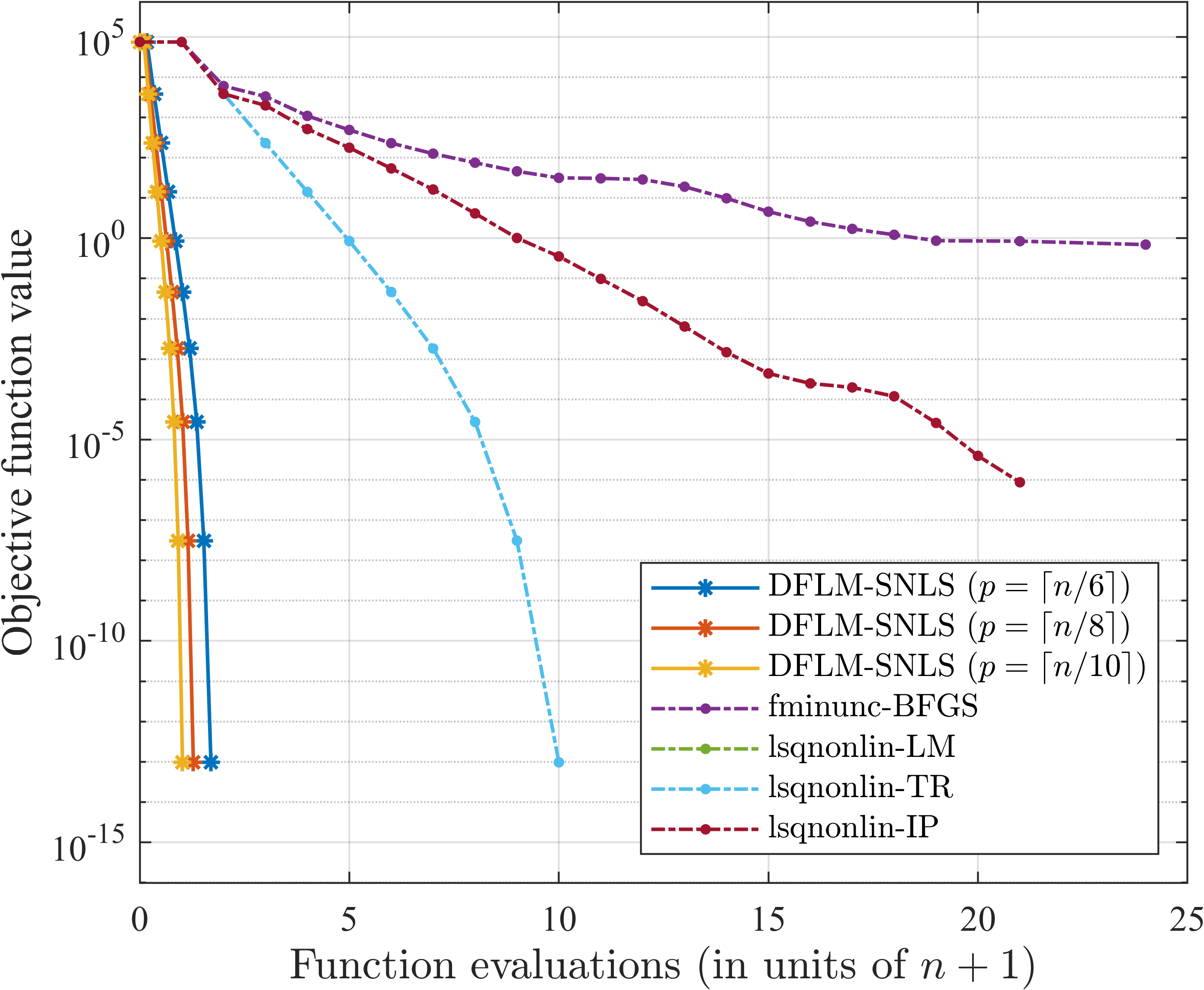}}
	\caption{Results for Example \ref{ex:4} with $n=100$ and $n=500$.}
	\label{fig: Luksan2018-24}
\end{figure}

\begin{example}
	We also test the algorithms on 83 sparse nonlinear least squares problems with $ n \approx 100 $,
	including problems 21, 22, 28, 30, 31 in \cite{more1981TestingUnconstrainedOptimization},
	problems 1--60, 69, 72, 79, 81, 82, 83, 84 in \cite[Section 2]{luksan2018ProblemsNonlinearLeast},
	problems 13, 51, 61, 63, 64 in \cite[Section 3]{luksan2018ProblemsNonlinearLeast}
	and problems 1, 2, 4, 5, 7, 8 in  \cite{friedlander1997SolvingNonlinearSystems}.

\end{example}


We present the numerical results by using the performance profile \cite{dolan2002BenchmarkingOptimizationSoftware}.
Denote by $ \mathcal{S} $ and $ \mathcal{P} $ the sets of algorithms and test problems, respectively.
If a solver $s\in \mathcal{S} $ gives a point $x$ that satisfies
\begin{equation}\label{eq: numerical convergence}
	f(x)\leq \tau f(x_0)+(1-\tau)f_p^*,
\end{equation}
where $ \tau \in(0, 1) $ is an accuracy level and $f_p^*$ is the smallest value of the objective function of problem $p$ obtained by any solver in $ \mathcal{S} $,
we say that $ s\in \mathcal{S} $ solves $ p\in \mathcal{P} $ up to the convergence test.

Let $ N_{s,p} $ be the least number of function evaluations of $F(x)$ required by the solver $s\in \mathcal{S} $ to solve the problem $p\in \mathcal{P}$ up to the convergence test.
If $ s $ fails to satisfy \eqref{eq: numerical convergence} for $ p $ within $ 1000(n+1) $ function evaluations, we let $ N_{s,p}=\infty$.
The performance profile of $s$ is defined as
\begin{equation}
	\pi_s(\alpha) := \frac{\card\left(\Big\{p\in\mathcal{P}\mid N_{s,p}\leq \alpha\cdot\min_{s\in\mathcal{S}}\{N_{s,p}\}\Big\}\right)}{\card(\mathcal{P})}
\end{equation}
for $\alpha\ge1$, where $ \card(\cdot) $ is the cardinality of a set.
Hence, $ \pi_s(\alpha) $ is the proportion of problems that are solved by $ s $ with function evaluations at most $ \alpha $ times of $\min_{s\in\mathcal{S}}\{N_{s,p}\} $.
In particular, $ \pi_s(1) $ is the proportion of problems that $s$ performs better than any other solvers in $ \mathcal{S} $
and $ \lim_{\alpha\rightarrow\infty}\pi_s(\alpha)$ is the proportion of problems that can be solved by $ s $.

Figures \ref{fig: More1981-30}--\ref{fig: Luksan2018-24} show that the choice of $ p $ has an important impact on the efficiency of the algorithm. A large $ p $ may increase unnecessary function evaluations, while a small $ p $ may lead to inaccurate Jacobian estimations. Hence, we develop an adaptive strategy for selecting $ p $: If $ d_k $ is accepted, then $ p $ is updated to $ \max\{p_{\min},\min\{p_{\max},p+p_{\mathrm{diff}}\}\} $; otherwise, $ p $ is updated to $ \max\{p_{\min},\min\{p_{\max},p-p_{\mathrm{diff}}\}\} $.
In the subsequent experiments, for the $ p $-adaptive version, we initially set $ p = \lceil n/3\rceil $, and take $ p_{\min} = \lceil n/4\rceil $, $ p_{\max} = \lceil n/2\rceil $ and $ p_{\mathrm{diff}} = \lceil n/10\rceil $.

The results are given in Figure \ref{fig: performance profile}, where $\alpha$ is displayed in $\log2$-scale and the tolerance $ \tau $ are taken as $ 10^{-2} $, $ 10^{-4} $, $ 10^{-6} $ and $ 10^{-8} $, respectively.
As shown, DFLM-SNLS almost always performs better than other algorithms.
DFLM-SNLS ($ p=\lceil n/4\rceil $) performs best on at least 60\% of the problems for all $ \tau $,
	followed by DFLM-SNLS ($ p$-adaptive), which performs best on at least 35\% of the problems.
Note that a larger $ \alpha $ allows for more function evaluations.
We observe that DFLM-SNLS ($ p=\lceil n/2\rceil $), DFLM-SNLS ($ p=\lceil n/3\rceil $)  and DFLM-SNLS ($ p$-adaptive)  can solve a slightly higher proportion of problems than DFLM-SNLS ($ p=\lceil n/4\rceil $) as $ \alpha $ increases,
and DFLM-SNLS can solve almost all problems when $ \alpha $ is larger.

\begin{figure}[htbp]
	\captionsetup{justification=raggedright, singlelinecheck=false}
	\centering
	\subfigure[$\tau=10^{-2} $]{
	\includegraphics[width=0.45\linewidth]{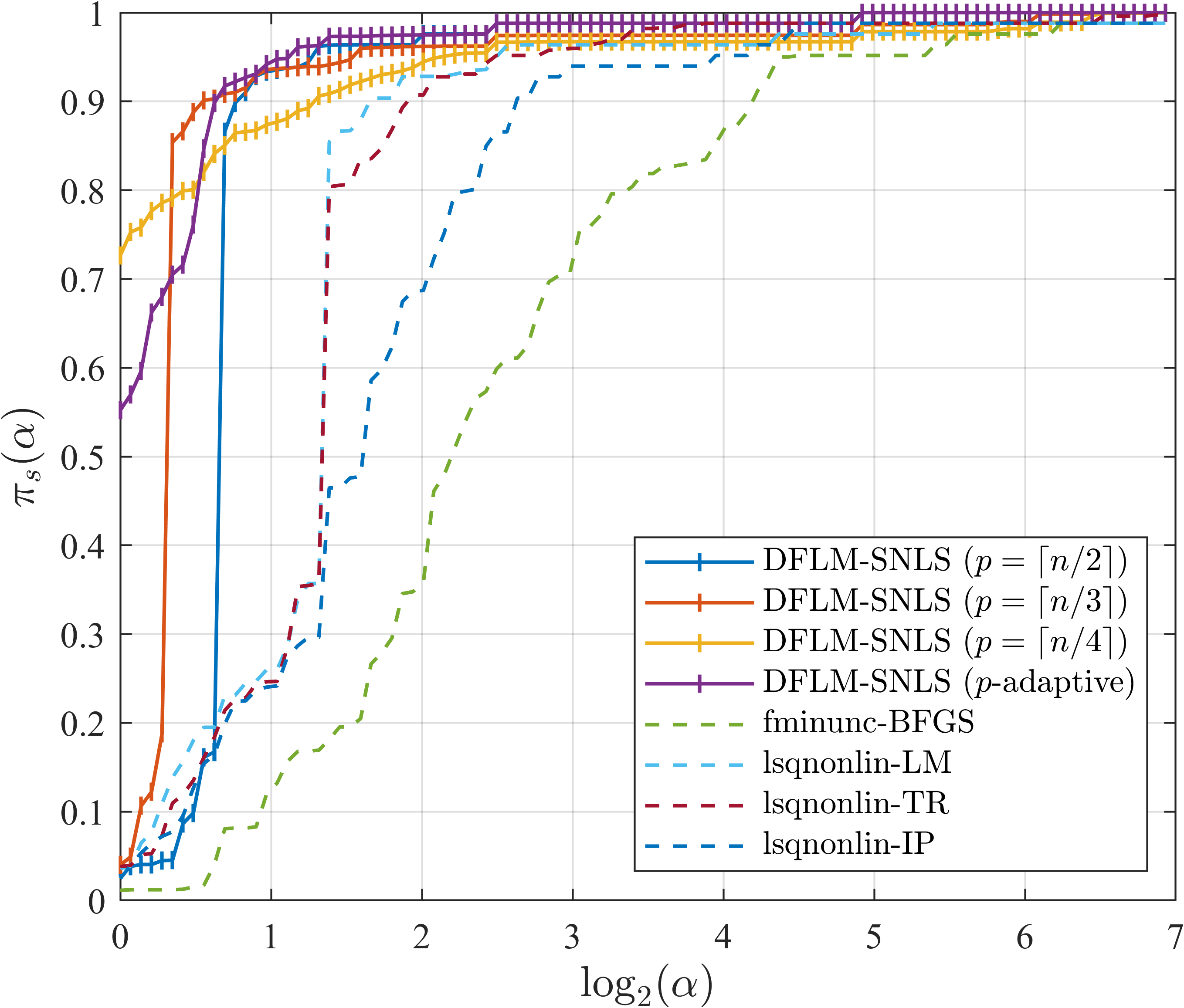}}\quad
	\subfigure[$\tau=10^{-4} $]{
	\includegraphics[width=0.45\linewidth]{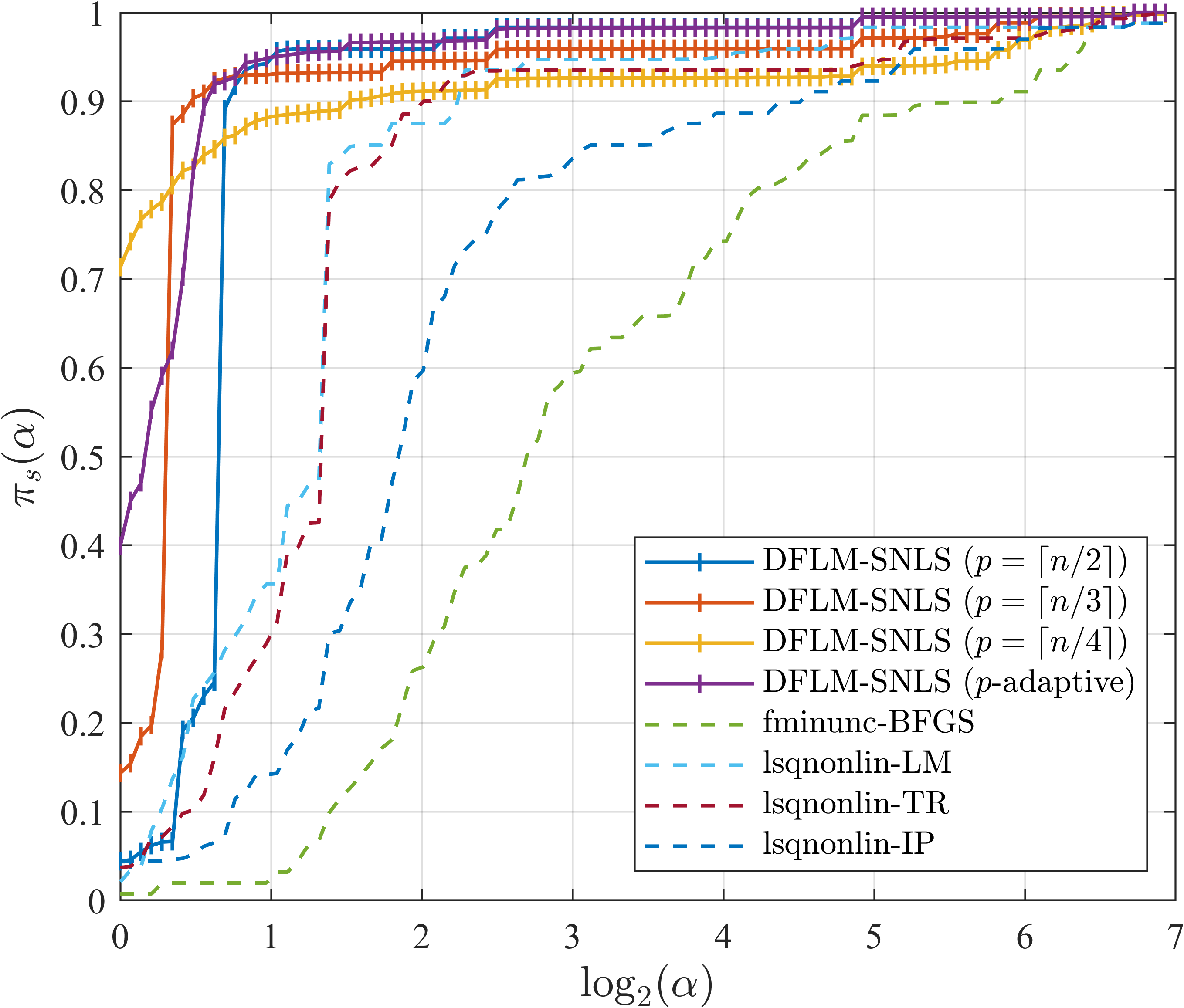}}
	\\
	\subfigure[$\tau=10^{-6} $]{
	\includegraphics[width=0.45\linewidth]{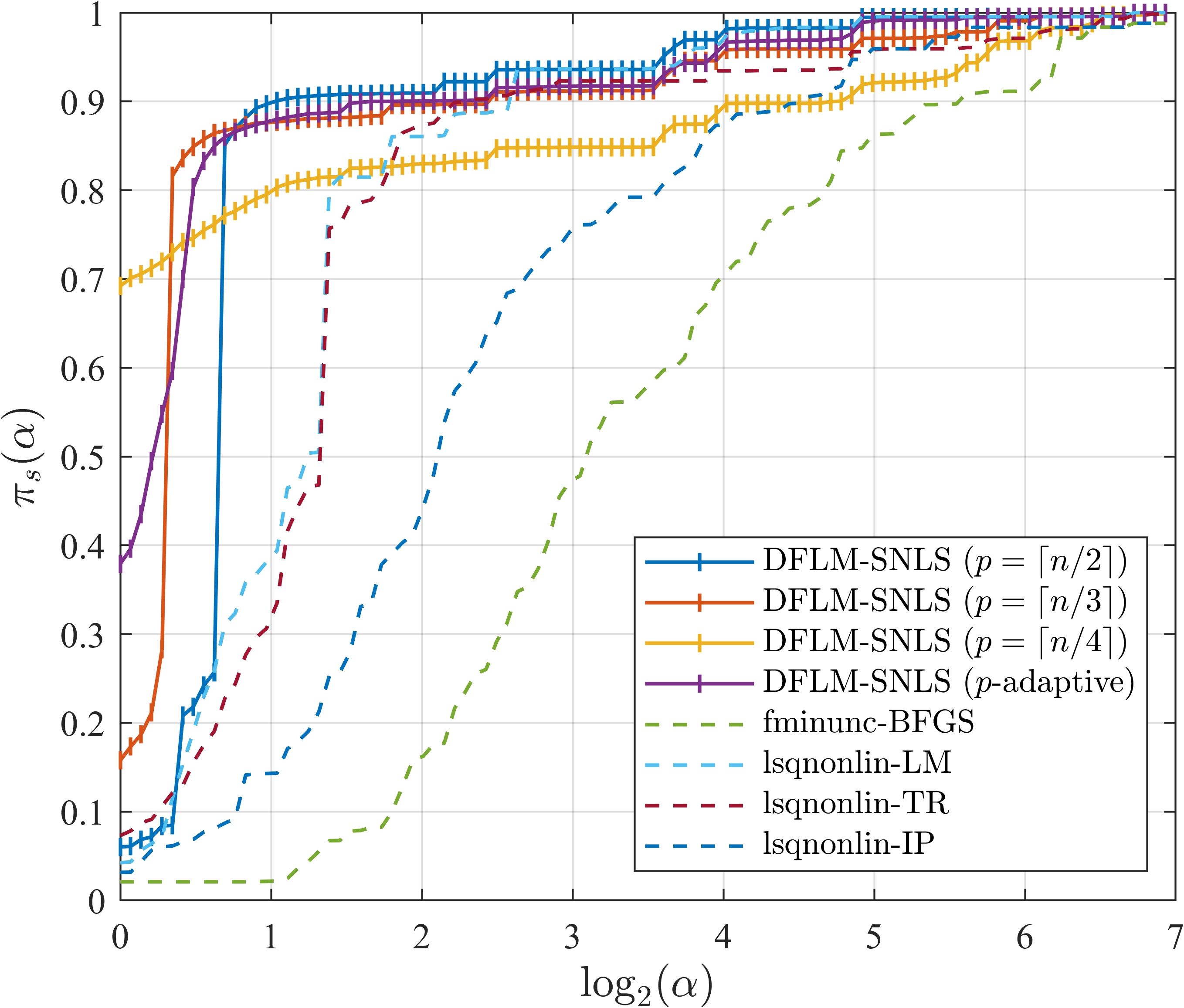}}\quad
	\subfigure[$\tau=10^{-8} $]{
	\includegraphics[width=0.45\linewidth]{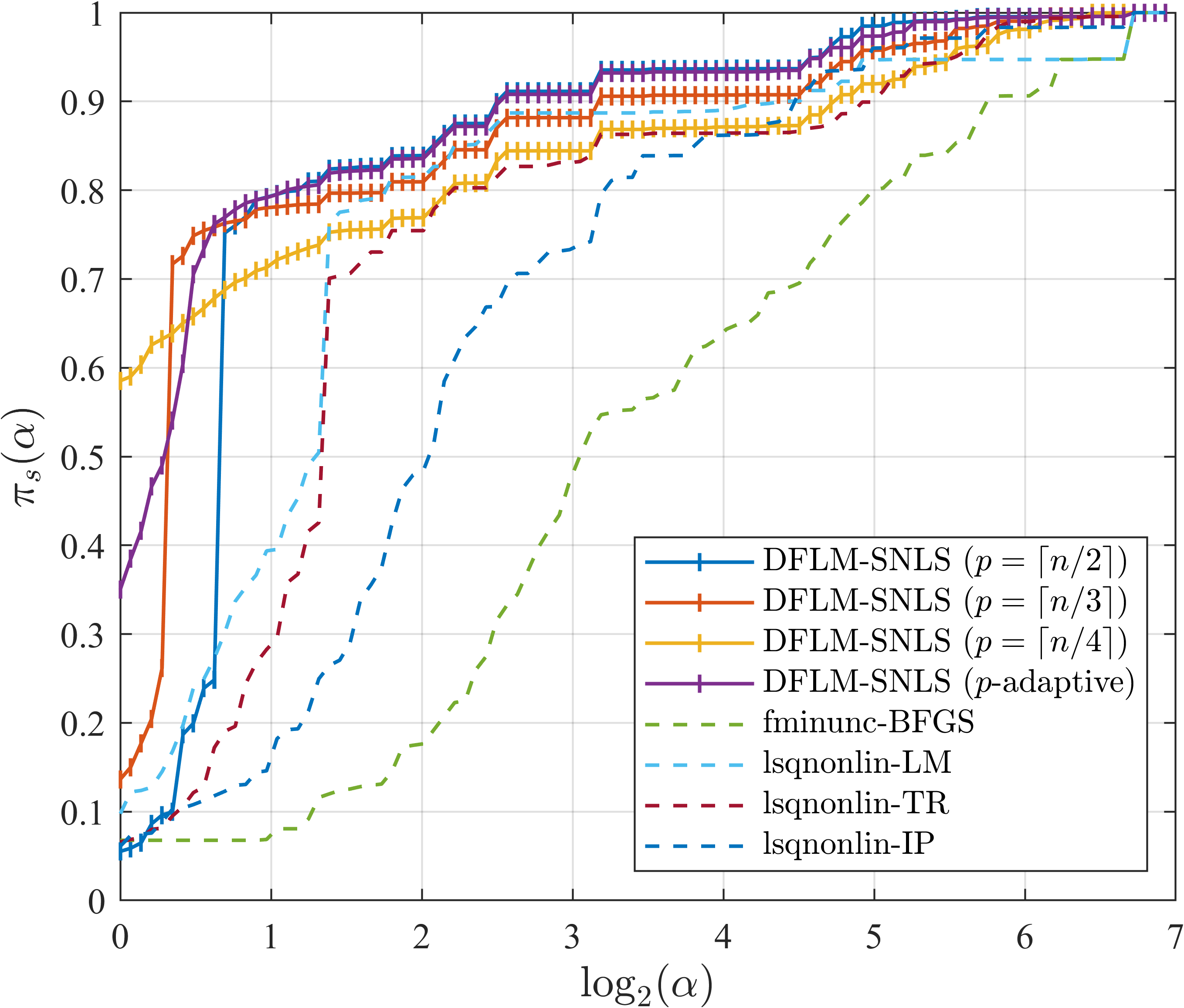}}
	\caption{Performance profiles on sparse nonlinear least squares problems in different accuracy levels}
	\label{fig: performance profile}
\end{figure}

\end{document}